\documentclass[12pt]{amsart}
\usepackage{graphicx}
\usepackage{subfigure}
\usepackage{eepic}
\usepackage{xcolor}
\selectcolormodel{gray}
\usepackage{tikz}
\usepackage{amsmath}
\usepackage{a4wide}
\usepackage[utf8]{inputenc}
\usepackage{amssymb}
\usepackage{amsopn}
\usepackage{epsfig}
\usepackage{amsfonts}
\usepackage{latexsym}
\usepackage{amsthm}
\usepackage{enumerate}
\usepackage[UKenglish]{babel}
\usepackage{verbatim}
\usepackage{color}

\newtheorem{theorem}{Theorem}[section]
\newtheorem{lemma}[theorem]{Lemma}
\newtheorem{proposition}[theorem]{Proposition}

\theoremstyle{definition}

\newtheorem{example}[theorem]{Example}

\theoremstyle{remark}

\numberwithin{equation}{section}

\renewcommand{\b}{\beta}

\renewcommand{\u}{\mathcal{U}}

\newcommand{\set}[1]{\left\{#1\right\}}

\newcommand{\f}{\infty}

\newcommand{\om}{\omega}

\newcommand{\al}{\alpha}

\newcommand{\si}{\sigma}
\renewcommand{\th}{\theta}
\newcommand{\ra}{\rightarrow}

\begin{document}

\title{Numbers with simply normal $\beta$-expansions}

\author{Simon Baker}
\address{Mathematical Institute, University of Warwick, Coventry, CV4 7AL, UK}
\email{simonbaker412@gmail.com}

\author{Derong Kong}
\address{Mathematical Institute, University of Leiden, PO Box 9512, 2300 RA Leiden, The Netherlands}
\email{d.kong@math.leidenuniv.nl }

\date{\today}


\begin{abstract}
In \cite{Bak} the first author proved that for any $\beta\in (1,\beta_{KL})$ every $x\in(0,\frac{1}{\beta-1})$ has a simply normal $\beta$-expansion, where $\beta_{KL}\approx 1.78723$ is the Komornik-Loreti constant. This result is complemented by an observation made in \cite{JSS}, where it was shown that whenever $\beta\in (\beta_T, 2]$ there exists an $x\in(0,\frac{1}{\beta-1})$ with a unique $\beta$-expansion, and this expansion is not simply normal. Here $\beta_T\approx 1.80194$ is the unique zero in $(1,2]$ of the polynomial $x^3-x^2-2x+1$.  This leaves a gap in our understanding within the interval $[\beta_{KL}, \beta_T]$. In this paper we fill this gap and prove that for any $\beta\in (1,\beta_T],$ every $x\in(0,\frac{1}{\beta-1})$ has a simply normal $\beta$-expansion. For completion, we provide a proof that for any $\beta\in(1,2)$, Lebesgue almost every $x$ has a simply normal $\beta$-expansion. We also give examples of $x$ with multiple $\beta$-expansions, none of which are simply normal. 

Our proofs rely on ideas from combinatorics on words and dynamical systems.
\end{abstract}

\maketitle

\section{Introduction}\label{sec:1}

Expansions in non-integer bases were first introduced and studied in the papers of Parry \cite{Parry} and R\'{e}nyi \cite{Renyi}. These representations are obtained by taking the usual integer base representations of the positive real numbers, and replacing the base by some non-integer. Despite being a simple generalisation of an idea that is well known to high school students, these representations exhibit many fascinating properties. One of these properties is the fact that typically a number has infinitely many representations. Consequently, one might ask whether amongst the set of representations there exists an expansion that satisfies a certain additional property. Properties we might be interested in could be combinatorial, number theoretic, or statistical. These ideas motivate this paper, wherein we study the existence of an expansion satisfying a certain statistical property, namely being simply normal.

Let $\beta\in(1,2]$ and $I_{\beta}:=[0,\frac{1}{\beta-1}]$. Given $x\in I_{\beta}$ we call a sequence $(\epsilon_i)\in\{0,1\}^{\mathbb{N}}$ a $\beta$-expansion of $x$ if $$x=\pi_{\beta}((\epsilon_i)):=\sum_{i=1}^{\infty}\frac{\epsilon_i}{\beta^i}.$$ It is a straightforward exercise to show that every $x\in I_{\beta}$ has at least one $\beta$-expansion. When $\beta=2$  then modulo a countable set every $x\in[0,1]$ has a unique binary expansion. Moreover, within this exceptional set every $x$ has precisely two expansions. However, when $\beta\in(1,2)$ the situation is very different. Below we recall some results that exhibit these differences.

\begin{enumerate}
\item Let $\beta\in(1,\frac{1+\sqrt{5}}{2}).$ Then every $x\in (0,\frac{1}{\beta-1})$ has a continuum of $\beta$-expansions \cite{Erdos}.
\item Let $\beta\in(1,2)$. Then Lebesgue almost every $x\in I_{\beta}$ has a continuum of $\beta$-expansions \cite{DajVri07, Sid2}.
\item For any $k\in \mathbb{N}\cup\{\aleph_{0}\}$ there exist $\beta\in(1,2)$ and $x\in I_{\beta}$ with exactly $k$ $\beta$-expansions \cite{Bak2,BakerSid, EHJ, EJ,KomKon, Sid1}.
\end{enumerate}We emphasise that the endpoints of $I_{\beta}$ have a unique $\beta$-expansion for any $\beta\in(1,2)$. Consequently, most of the statements we make will relate to its interior $(0,\frac{1}{\beta-1}).$

Given a sequence $(\epsilon_{i})\in\{0,1\}^{\mathbb{N}},$ we define the \emph{frequency of zeros of $(\epsilon_i)$} to be the limit $$\textrm{freq}_{0}(\epsilon_i):=\lim_{n\to\infty}\frac{\#\{1\leq i \leq n: \epsilon_i=0\}}{n}.$$
Assuming the limit exists. Where $\# A$ denotes the cardinality of a set $A$.  We say that $(\epsilon_i)$ is \emph{simply normal} if $\textrm{freq}_{0}(\epsilon_i)=1/2$. In \cite{Bak} the first author proved the following theorem.

\begin{theorem}
\label{frequency theorem}
\begin{enumerate}
  \item Let $\beta\in(1,\beta_{KL})$. Then every $x\in(0,\frac{1}{\beta-1})$ has a simply normal $\beta$-expansion.
  \item Let $\beta\in(1,\frac{1+\sqrt{5}}{2})$. Then every $x\in(0,\frac{1}{\beta-1})$ has a $\beta$-expansion for which the frequency of zeros does not exist.
  \item Let $\beta\in(1,\frac{1+\sqrt{5}}{2})$. Then there exists $c=c(\beta)>0$ such that for every $x\in (0,\frac{1}{\beta-1})$ and $p\in[1/2-c,1/2+c],$ there exists a $\b$-expansion of $x$ with frequency of zeros equal to $p$.
\end{enumerate}
\end{theorem}
The quantity $\beta_{KL}\approx 1.78723$ appearing in statement $(1)$ of Theorem \ref{frequency theorem} is the \emph{Komornik-Loreti constant} introduced in \cite{KomLor}. Both statements $(2)$ and $(3)$ appearing in Theorem \ref{frequency theorem} are sharp. For any $\beta\in[\frac{1+\sqrt{5}}{2},2),$ there exists an $x\in (0,\frac{1}{\beta-1})$ such that for any $\beta$-expansion of $x$ its frequency of zeros exists and is equal to either $0$ or $1/2$. It is natural to wonder whether the parameter space described in statement $(1)$ of Theorem \ref{frequency theorem} is optimal. In \cite{JSS} Jordan, Shmerkin, and Solomyak proved the following result.

\begin{theorem}
\label{JSS theorem}
If $\beta\in(\beta_T,2].$ Then there exists $x\in (0,\frac{1}{\beta-1})$ with a unique $\beta$-expansion, and this expansion is not simply normal.
\end{theorem} Here $\beta_{T}\approx 1.80194.$ We will elaborate more on how $\beta_T$ and $\beta_{KL}$ are defined later. Theorem \ref{frequency theorem} and Theorem \ref{JSS theorem}  leave an interval $[\beta_{KL}, \beta_T]$ for which we do not know whether every $x\in(0,\frac{1}{\beta-1})$ has a simply normal $\beta$-expansion. In this paper we fill this gap and prove the following theorem.

\begin{theorem}
\label{Main theorem}
Let $\beta\in (1,\beta_T].$ Then every $x\in (0,\frac{1}{\beta-1})$ has a simply normal $\beta$-expansion.
\end{theorem} With Theorem \ref{JSS theorem} in mind it is natural to ask whether it is possible for an $x$ to have multiple $\beta$-expansions, none of which are simply normal. In this paper we include several explicit examples which demonstrate that this behaviour is possible.

The rest of this paper is arranged as follows. In Section \ref{Preliminaries} we recall some necessary preliminaries. We prove Theorem \ref{Main theorem} in Section \ref{proofs}. We conclude in Section \ref{examples} with our aforementioned examples, and we also provide a short proof that for any $\beta\in(1,2)$, Lebesgue almost every $x\in I_\beta$ has a simply normal $\beta$-expansion. At the end of the paper we   pose some questions.

\section{Preliminaries}
\label{Preliminaries}
The proof of Theorem \ref{Main theorem} will make use of a dynamical interpretation of $\beta$-expansions, along with some properties of unique expansions. We start by detailing the relevant dynamical preliminaries.

\subsection{Dynamical preliminaries}
Given $\beta\in(1,2)$ and $x\in I_{\beta},$ we denote the set of $\beta$-expansions of $x$ as follows $$\Sigma_{\beta}(x):=\Big\{(\epsilon_i)\in\{0,1\}^{\mathbb{N}}: x=\sum_{i=1}^{\infty}\frac{\epsilon_i}{\beta^i}\Big\}.$$ Now let us fix the maps $T_0(x)=\beta x$ and $T_1(x)=\beta x-1$. Notice that the maps $T_0$ and $T_1$ depend on the parameter $\beta$.  Given $\beta\in(1,2)$ and $x\in I_{\beta},$ let $$\Omega_{\beta}(x):=\Big\{(a_i)\in\{T_0,T_1\}^{\mathbb{N}}:(a_n\circ\cdots \circ a_1)(x)\in I_{\beta} \textrm{ for all }n\in\mathbb{N}\Big\}.$$ The following lemma was proved in \cite{BakG} (see also, \cite{DajVri05}). It shows how one can interpret a $\beta$-expansion dynamically as a sequence of maps that do not map a point out of $I_{\beta}$.

\begin{lemma}
\label{Bijection lemma}
For any $x\in I_{\beta}$ we have $\textrm{Card }\Sigma_{\beta}(x)=\textrm{Card }\Omega_{\beta}(x)$. Moreover, the map which sends $(\epsilon_i)$ to $(T_{\epsilon_i})$ is a bijection between $\Sigma_{\beta}(x)$ and $\Omega_{\beta}(x)$.
\end{lemma}
We refer the reader to Figure \ref{fig1} for a graph of the functions $T_0$ and $T_1$. One observes that these graphs overlap on the interval $$S_{\beta}:=\Big[\frac{1}{\beta},\frac{1}{\beta(\beta-1)}\Big].$$ If $x\in S_{\beta}$ then both $T_0$ and $T_1$ map $x$ back into $I_{\beta}$. In which case, by Lemma \ref{Bijection lemma}, $x$ has a $\beta$-expansion that begins with a $0$ and a $\beta$-expansion that begins with a $1$. More generally, if $x$ can be mapped into $S_{\beta}$ under a finite sequence of $T_0$'s and $T_1$'s, then $x$ has at least two $\beta$-expansions. In the literature $S_{\beta}$ is commonly referred to as the \emph{switch region}. An understanding of how orbits are mapped into $S_{\beta},$ and how orbits can avoid $S_{\beta},$ often proves to be profitable when studying a variety of problems. The main technical innovation of this paper is Proposition \ref{covering prop}, which gives a thorough description of how orbits are mapped into $S_{\beta}$.

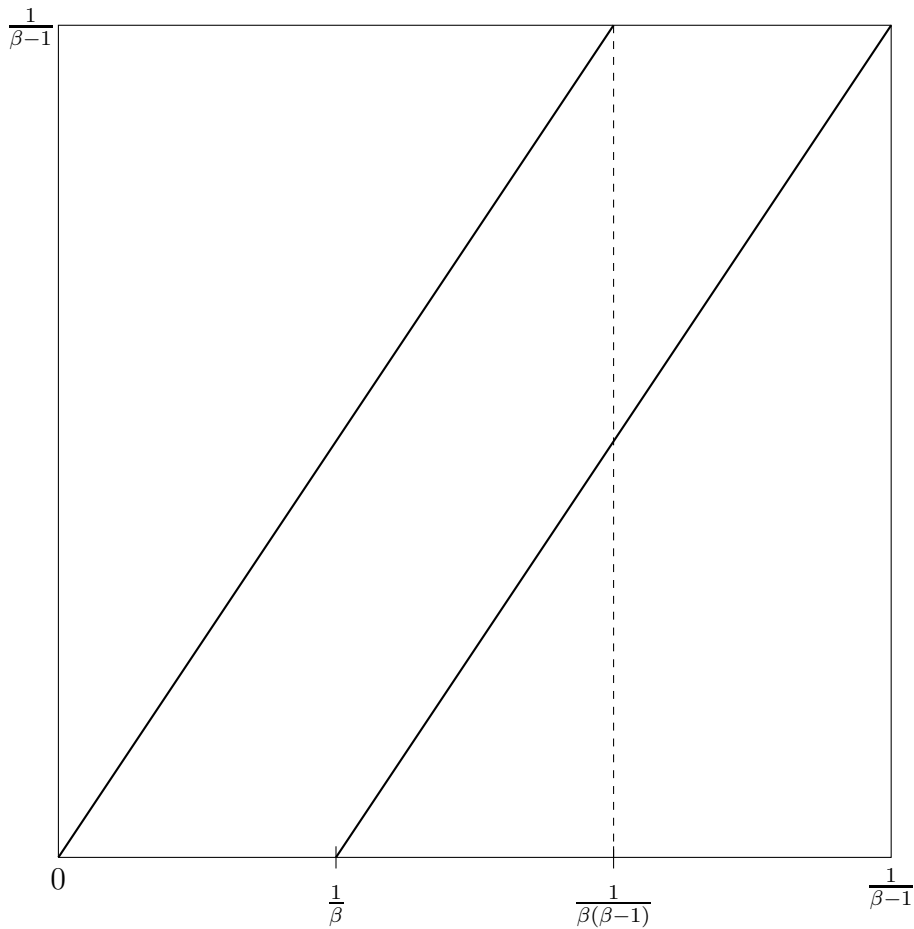
\begin{figure}
\centering
\begin{tikzpicture}[x=2.1,y=2.1]
\path[draw](0,10) -- (0,160) -- (150,160) -- (150,10) -- (0,10);
\path[draw][thick](0,10) -- (100,160);
\path[draw][dashed](100,160) -- (100,10);
\path[draw][thick](50,10) -- (150,160);
\path[draw](50,8) -- (50,12);
\path[draw](100,8) -- (100,12);
\draw (0,10) node[below] {$0$};
\draw (50,7) node[below] {$\frac{1}{\beta}$};
\draw (100,7) node[below] {$\frac{1}{\beta(\beta-1)}$};
\draw (150,10) node[below] {$\frac{1}{\beta-1}$};
\draw (-5,165) node[below] {$\frac{1}{\beta-1}$};
\end{tikzpicture}
\caption{The overlapping graphs of $T_0$ and $T_1$.}
 \label{fig1}

\end{figure}

By Lemma \ref{Bijection lemma}, one can reinterpret Theorem \ref{Main theorem} in terms of the existence of a sequence of maps with limiting frequency of $T_0$'s equal to $1/2$. We make use of this interpretation in our proof. With this in mind we introduce the following notation. Let $\{T_0,T_1\}^*:=\cup_{n=1}^{\infty}\{T_0,T_1\}^n$.
 Given $a\in \{T_0,T_1\}^*$ let $|a|$ denote the length of $a.$ Moreover, given $a\in\{T_0,T_1\}^*$ let $$|a|_0:=\#\{1\leq i \leq |a|: a_i=T_0\}$$ and
  $$|a|_1:=\#\{1\leq i \leq |a|: a_i=T_1\}.$$ We will use the same notation to denote the analogous quantities for finite sequences of zeros and ones. Whether we are referring to a finite sequence of maps or a finite sequence of zeros and ones should be clear from the context.

It is useful at this point to introduce the following interval. Given $\beta\in(1,2),$ let
 $$O_{\beta}:=[\pi_\beta((01)^\f), \pi_\beta((10)^\f)]=\Big[\frac{1}{\beta^2-1},\frac{\beta}{\beta^2-1}\Big].$$
 Here and throughtout we use $\om^\f$ to denote the element of $\set{0, 1}^{\mathbb N}$ obtained by infinitely concatenating a finite sequence $\om$.  Notice that $T_{0}(\frac{1}{\beta^2-1})=\frac{\beta}{\beta^2-1}$ and $T_{1}(\frac{\beta}{\beta^2-1})=\frac{1}{\beta^2-1}.$ What is more, $T_0$ and $T_1$ expand distances between points by a factor $\beta,$ and have their unique fixed points at $0$ and $\frac{1}{\beta-1}$ respectively. It is a consequence of these observations that given $x\in (0,\frac{1}{\beta-1})\setminus O_{\beta},$ there exists $k\in\mathbb{N}$ and $i\in\{0,1\}$ such that $T_i^k(x)\in O_{\beta}$. Therefore all orbits are eventually mapped into $O_{\beta}$, and thus $O_{\beta}$ can be thought of as an attractor for this system.

\subsection{Univoque preliminaries}
A classical object of study within expansions in non-integer bases is the set of $x$ with a unique expansion. Fixing notation, given $\beta\in (1,2)$ let
$${U}_{\beta}:=\Big\{x\in\Big[0,\frac{1}{\beta-1}\Big]:x \textrm{ has a unique }\beta\textrm{-expansion}\Big\}$$and
 $$\widetilde{{U}}_{\beta}:=\Big\{(\epsilon_i)\in\{0,1\}^{\mathbb{N}}:\sum_{i=1}^{\infty}\frac{\epsilon_i}{\beta^i}\in\mathcal{U}_{\beta}\Big\}.$$
We call ${U}_{\beta}$ the \emph{univoque set} and $\widetilde{{U}}_{\beta}$ the set of \emph{univoque sequences}. By definition there is a bijection between these two sets. For more on these sets we refer the reader to \cite{BarBakKong, deVKom,KomKonLi} and the survey papers \cite{VriKom16, Kom11}.

The lexicographic ordering on $\{0,1\}^{\mathbb{N}}$ is a useful tool for studing the univoque set. This ordering is defined as follows. Given $(\epsilon_i), (\delta_i)\in \{0,1\}^{\mathbb{N}}$ we say that $(\epsilon_i)\prec (\delta_i)$ if $\epsilon_1<\delta_1,$ or if there exists $n\in\mathbb{N}$ such that $\epsilon_{n+1}<\delta_{n+1}$ and $\epsilon_i=\delta_i$ for all $1\leq i \leq n$. We define $\preceq, \succ, \succeq$ in the obvious way. These definitions also have the obvious interpretation for finite sequences. We define the \emph{reflection} of  a word by $\overline{\epsilon_1\ldots \epsilon_n}=(1-\epsilon_1)\cdots(1-\epsilon_n)$, and the reflection of a sequence by $\overline{(\epsilon_i)}=(1-\epsilon_1)(1-\epsilon_2)\cdots$.

Many properties of $\widetilde{{U}}_{\beta}$ and consequently ${U}_{\beta}$ are encoded in the quasi-greedy expansion of $1$. The \emph{quasi-greedy} expansion of $1$ is the lexicographically largest $\beta$-expansion of $1$ that does not end in $0^\f$ (cf.~\cite{DarKat95}). Given a $\beta\in(1,2)$ we denote the quasi-greedy expansion of $1$ by $\alpha(\beta)=(\alpha_i(\beta))$. The following description of $\al(\beta)$ is well-known (cf.~\cite{KomLor1}).
\begin{lemma}\label{lem:alpha-beta}
The map $\beta\mapsto \al(\beta)$ is a strictly increasing bijection between the interval $(1, 2]$ and the set of sequence $(\gamma_i)\in\set{0,1}^{\mathbb N}$ not ending with $0^\f$ and satisfying
\[
\gamma_{n+1}\gamma_{n+2}\ldots\preccurlyeq \gamma_1\gamma_2\ldots\quad\textrm{for all}\quad n\ge 0.
\]
Furthermore, the map $\beta\mapsto \al(\beta)$ is left continuous with respect to the order topology on $\set{0, 1}^{\mathbb N}$ induced by the metric $\rho((\epsilon_i), (\delta_i))=2^{-\inf\set{j\ge 1: c_j\ne d_j}}$.
\end{lemma}

Based on the notation $\al(\beta)$ we give  the     lexicographical characterization of $\widetilde{U}_\beta$ (cf.~\cite{deVKom}).
\begin{lemma}
\label{lem:univoque}
Let $\beta\in(1,2]$. Then $(\epsilon_i)\in\widetilde U_\beta$ if and only if the sequence $(\epsilon_i)$ satisfies
\begin{align*}
\epsilon_{n+1}\epsilon_{n+2}\ldots \prec \al(\beta)&\quad\textrm{whenever}\quad \epsilon_n=0,\\
\epsilon_{n+1}\epsilon_{n+2}\ldots\succ \overline{\al(\beta)}&\quad\textrm{whenever}\quad \epsilon_n=1.
\end{align*}
\end{lemma}
Note by Lemma \ref{lem:alpha-beta}  that the map $\beta\mapsto \al(\beta)$ is  strictly increasing. Then by Lemma \ref{lem:univoque}  it follows    that $\widetilde U_{\beta_1}\subseteq\widetilde U_{\beta_2}$ whenever $\beta_1<\beta_2$.

The aforementioned constants $\beta_{KL}$ and $\beta_{T}$ are defined by their quasi-greedy expansions. The Komornik-Loreti constant $\beta_{KL}$ is the unique $\beta\in(1, 2)$ whose quasi-greedy expansion is the shifted Thue-Morse sequence. This sequence is defined as follows. Let $\tau^0=0$, we define $\tau^1$ to be $\tau^0$ concatenated with its reflection, in other words $\tau^1=\tau^0\overline{\tau^0}.$ We then define $\tau^2$ to be the concatentation of $\tau^1$ with its reflection. We repeat this process in the natural way, given $\tau^k$ let $\tau^{k+1}$ be the concatenation of $\tau^k$ with its reflection. The first few words built using this procedure are listed below
$$\tau^{0}=0,\quad \tau^{1}=01,\quad \tau^2=0110,\quad \tau^3=01101001.$$
Repeating this reflection and concatenation process indefinitely gives rise to an infinite sequence $ (\tau_i)_{i=0}^{\infty}$. This sequence is called the \emph{Thue-Morse sequence}. The Komornik-Loreti constant $\beta_{KL}$ satisfies $\alpha(\beta_{KL})=(\tau_i)_{i=1}^{\infty}.$ The Komornik-Loreti constant first appeared in  \cite{KomLor} where it was shown to be the smallest $\beta\in(1,2)$ for which $1$ has a unique $\beta$-expansion. It has since been shown to be important for a variety of other reasons, see \cite{GlenSid}. In \cite{AllCos} it was shown that $\beta_{KL}$ is transcendental. For more on the Thue-Morse sequence we refer the reader to \cite{AllShall}.

The quantity $\beta_{T}$ is the unique $\beta\in(1,2)$ such that $\alpha(\beta)=1(10)^{\infty}$.    Alternatively, $\beta_T$ is the unique root of $x^3-x^2-2x+1=0$ that lies within the interval $(1,2)$. We emphasise here that $\beta_T$ is not a Pisot number. $\beta_T$ although not as exotic as $\beta_{KL}$ is still of importance when it comes to studying $U_{\beta}$ and $\widetilde{U}_{\beta}$. $\beta_{T}$ is the smallest $\beta\in(1,2)$ for which the attractor of $\widetilde{U}_{\beta}$ is transitive under the usual shift map, see \cite{Rafa,BarBakKong}. Moreover, it is a consequence of the work done in \cite{AllSid} that $\widetilde{U}_{\beta}$ contains a periodic orbit of odd length if and only if $\beta\in(\beta_T,2)$. Observe that this result in fact implies Theorem \ref{JSS theorem}. For completion we provide a short proof that if $\beta\in(\beta_T,2)$ then $\widetilde{U}_{\beta}$ contains a periodic orbit of odd length. For any $j\in \mathbb{N}$ there exists $\beta_j\in(\beta_T,2)$ such that $\alpha(\beta_j)=(1(10)^j)^{\infty}.$ It can be shown that $\beta_j\searrow \beta_T$ as $j\to\infty.$ It follows from an application of Lemma \ref{lem:univoque} that for any $j\in\mathbb{N}$ the sequence $(1(10)^{j+1})^{\infty}$ is contained in $\widetilde{U}_{\beta_j}$. Notice that the periodic block of $(1(10)^{j+1})^{\infty}$ has odd length. Now for any $\beta\in(\beta_T,2)$ there exists $\beta_j\in(\beta,\beta_T)$, so by our previous observation and the fact that $\widetilde{U}_{\beta_j}\subseteq \widetilde{U}_{\beta},$ it follows that $(1(10)^{j+1})^{\infty}\in\widetilde{U}_{\beta}.$

In \cite{JSS} the following useful technical result was proved.

\begin{lemma}
\label{unique}
Let $\beta\in(1,\beta_{T}]$. If $(\epsilon_i)\in\widetilde{U}_{\beta}\setminus\{0^{\infty},1^{\infty}\}$ then $(\epsilon_i)$ is simply normal.
\end{lemma}

In our proofs we will also require the notion of a Thue-Morse chain and a Thue-Morse interval. We define these now. Let $\omega^0\in \{0,1\}^*$ be a finite word beginning with zero. We then let $\omega^1=\omega^0\overline{\omega^0}$. More generally, suppose that $\omega^k$ has been defined for some $k\in\mathbb{N}$. We then let $\omega^{k+1}=\omega^k\overline{\omega^k}$. Note that $|\omega^k|\to \infty$ as $k\to\infty$ and $\omega^{k+1}$ coincides with $\omega^k$ in the first $|\omega^k|$ entries. Consequently, we can consider the componentwise limit of the sequence $(\omega^k).$ We denote this infinite sequence by $\omega^{TM}$. We call the sequence $(\omega^k)$ a \emph{Thue-Morse chain}. The Thue-Morse sequence is obtained by taking $\omega^0=0$. In this case $\omega^{TM}=(\tau_i)_{i=0}^{\infty}.$
Given a $\beta\in(1,2)$ and a Thue-Morse chain $(\omega^k)$,  we say that the interval
$$I_{\omega^0}:=[\pi_{\beta}((\omega^0)^{\infty}),\pi_{\beta}((\omega^{TM}))]$$ is a \emph{Thue-Morse interval} if the following inequalities hold:
$$\pi_{\beta}((\omega^0)^{\infty})<\pi_{\beta}((\omega^1)^{\infty})<\cdots< \pi_{\beta}((\omega^k)^{\infty})<\pi_{\beta}((\omega^{k+1})^{\infty})<\cdots< \pi_{\beta}(\omega^{TM}).$$ Similarly, we say that the interval
\[ J_{\omega^0}:=[\pi_\beta(\overline{\om^{TM}}), \pi_\beta((\overline{\om^0})^\f)]\]
 is a Thue-Morse interval if the following inequalities hold:
$$\pi_{\beta}(\overline{\omega^{TM}})<\cdots< \pi_{\beta}((\overline{\omega^{k+1}})^{\infty})< \pi_{\beta}((\overline{\omega^{k}})^{\infty})<\cdots < \pi_{\beta}((\overline{\omega^{1}})^{\infty})< \pi_{\beta}((\overline{\omega^{0}})^{\infty}).$$ Note that $I_{\omega^0}$ is a Thue-Morse interval if and only if $J_{\omega^0}$ is a Thue-Morse interval.
The following proposition will be used  to understand the possible itineraries of an $x$ that is mapped into the switch region.

\begin{proposition}
\label{covering prop}
For any  $\beta\in(1,\beta_T]$ there exists a set of words $\{\omega^\theta\}_{\theta\in\Theta}$ such that the following properties are satisfied:
\begin{enumerate}
\item For each $\theta\in\Theta$ the intervals $I_{\omega^\theta}$ and $J_{\omega^\theta}$ are  Thue-Morse intervals. Furthermore, the intervals $I_{\om^\th}, J_{\om^\th}$ with $\th\in\Theta$ are pairwise disjoint.
\item $$\Big[\frac{1}{\beta^2-1},\frac{1}{\beta}\Big)\setminus \bigcup_{\theta\in\Theta} I_{\omega^\theta}\subseteq U_{\beta}.$$
\item $$\Big(\frac{1}{\beta(\beta-1)},\frac{\beta}{\beta^2-1}\Big]\setminus \bigcup_{\theta\in\Theta} J_{\omega^\theta}\subseteq U_{\beta}.$$
\item Each $\omega^\theta$ satisfies $$\frac{\#\{1\leq i \leq |\omega^\theta|:\omega_i^\theta=0\}}{|\omega^\theta|}=\frac{1}{2}.$$
\item There exists $C>0$ such that for any $\theta\in \Theta$ and $1\leq n\leq |w^\theta|$
$$\Big|\#\{1\leq i \leq n: \omega^\theta_i=0\}-\#\{1\leq i \leq n: \omega^\theta_i=1\}\Big|\leq C.$$
\end{enumerate}
\end{proposition}
We remark that statements $(1),$ $(2),$ and $(3)$ in Proposition \ref{covering prop} in fact hold for any $\beta\in(1,2)$. Before proving this proposition we recall the following. Let $\u$ be the set of $\beta\in(1,2]$ such that $1\in U_\beta$. Then $\beta_{KL}=\min\u$ and  its topological closure $\overline{\u}$ is a Cantor set (cf.~\cite{KomLor1}). Furthermore,
\begin{equation}\label{eq:U}
\left [\frac{1+\sqrt{5}}{2}, 2\right]\setminus {\u}= \bigcup[\beta_0, \beta_*),
\end{equation}
where the union on the right hand side is countable and pairwise disjoint.  Indeed, even the closed intervals $[\beta_0, \beta_*]$ are pairwise disjoint. For each connected component $[\beta_0, \beta_*)\subset[\frac{1+\sqrt{5}}{2},  2]$ the left endpoint $\beta_0$ satisfies  that  $\al(\beta_0)$ is periodic, say $\al(\beta_0)=(\al_1\ldots \al_m)^\f$ with period $m$. Then $m\ge 2$ and $\al_m=0$. The right endpoint $\beta_*$ is called a \emph{de Vries-Komornik number} in \cite{KonLi} and satisfies $\beta_*\in\u$. The quasi-greedy expansion $\al(\beta_*)$ is a Thue-Morse type sequence defined as follows. Let $\al^0=\al_1\ldots \al_m$. Then we set $\al^1=\al_1\ldots \al_m^+\overline{\al_1\ldots\al_m^+}=(\al^0)^+\overline{(\al^0)^+}$. Here for a word $\epsilon_1\ldots \epsilon_n$ with $\epsilon_n=0$ we write $\epsilon_1\ldots \epsilon_n^+=\epsilon_1\ldots \epsilon_{n-1}(\epsilon_n+1)$. More generally, suppose $\al^k$ has been defined for some $k\ge 0$. Then we set $\al^{k+1}=(\al^k)^+\overline{(\al^k)^+}$. Thus, $\al(\beta_*)$ is the component-wise limit of the sequence $(\al^k)$. In this case $[\beta_0, \beta_*)$ is called the connected component generated by $\al^0=\al_1\ldots \al_m$ and denoted by $C_{\al_1\ldots \al_m}=C_{\al^0}$. Note by Lemma \ref{lem:alpha-beta}  that  for each $k\ge 1$ there exists a unique $\beta_k\in(\beta_0, \beta_*)$ such that $\al(\beta_k)=(\al^k)^\f$ (cf.~\cite{deVKom}). From the definition of $\al^k$ it follows that
\[\al(\beta_0)\prec\al(\beta_1)\prec\cdots\prec \al(\beta_k)\prec\al(\beta_{k+1})\prec\cdots\prec \al(\beta_*),\]
and $\al(\beta_k)$ converges to $\al(\beta_*)$ as $k\ra\f$. By
 Lemma \ref{lem:alpha-beta} this implies that
\begin{equation}\label{e21}
\beta_0<\beta_1<\cdots<\beta_k<\beta_{k+1}<\cdots<\beta_*\quad\textrm{\and}\quad \beta_k\nearrow \beta_*\textrm{ as }k\ra\f.
\end{equation}
 Observe that $\al(\frac{1+\sqrt{5}}{2})=(10)^\f$. Set $\al^0=10$. Then     $\al^1=1100, \al^2=11010010, \ldots,$ and  $\al(\beta_{KL})$ is the component-wise limit of the sequence $(\al^k)$. So the interval $[\frac{1+\sqrt{5}}{2}, \beta_{KL})$ is indeed the first connected component generated by $\al^0=10$, i.e., $C_{10}=[\frac{1+\sqrt{5}}{2}, \beta_{KL})$.

\begin{proof}[Proof of Proposition \ref{covering prop}]
Fix $\beta\in(1, \beta_T]$. Let $\set{\al^\th}_{\th\in\Theta}$ be the set of words such that for any $\th\in\Theta$  the  connected component $C_{\al^\th}=[\beta_0^{\th}, \beta_*^{ \th})$ intersects $[\frac{1+\sqrt{5}}{2}, \beta)$. Then by (\ref{eq:U}) it follows that
\begin{equation}\label{eq:intersectU}
\left[\frac{1+\sqrt{5}}{2}, \beta\right)\setminus\bigcup_{\theta\in\Theta} \overline{C_{\al^\theta}}\subseteq\left[\frac{1+\sqrt{5}}{2}, \beta\right)\cap\u,
\end{equation}
where $\overline{C_{\al^\theta}}=[\beta_0^\th, \beta_*^{ \th}]$ denotes the topological closure of $C_{\al^{\theta}}$. We emphasize that  the closed intervals $[\beta_0^{\th}, \beta_*^{ \th}], \th\in\Theta$ are pairwise disjoint.
 We first construct  for each connected component $C_{\al^\th}$ a unique Thue-Morse interval $I_{\om^\th}$.

Take $\th\in\Theta$ and let $C_{\al^\th}=[\beta^\th_0, \beta_*^{ \th})$ be the connected component generated by $\al^\th=\al^{\th,0}=\al_1\ldots \al_m$. Then $\al(\beta^\th_0)=(\al^{\th, 0})^\f=(\al_1\ldots \al_m)^\f$ with $\al_m=0$. Furthermore,   for each $k\ge 0$ let  $\al^{\th, k+1}$ be recursively defined by $\al^{\th, k+1}=(\al^{\th,k})^+\overline{(\al^{\th,k})^+}$. Then for any $k\geq 0$ there exists a unique $\beta^\th_k\in(\beta^\th_0, \beta_*^{ \th})$ such that  $\al(\beta^\th_k)=(\al^{\th,k})^\f$.  So we obtain a sequence of strictly increasing bases $(\beta^\th_k)$ as described in (\ref{e21}).  In the following  we construct the  Thue-Morse chain $(\om^{\th, k})$ in terms of the bases $(\beta^\th_k)$.

 Let $(\om^{\th, k})$ be  the Thue-Morse chain  generated by $\om^{\th}=\om^{\th,0}=\al_m\al_1\ldots \al_{m-1}$. We claim that
\begin{equation}\label{eq:omk-betak}
(\om^{\th, k})^\f=0\al(\beta^\th_k)
\end{equation}
for all $k\ge 0$.
We will prove the claim  by induction on $k$. First we consider $k=0$. Note that  $\al_m=0$. Then
\[(\om^{\th,0})^\f=0(\al_1\ldots \al_m)^\f=0(\al^{\th, 0})^\f=0\al(\beta^\th_0).\]
So (\ref{eq:omk-betak}) holds for $k=0$.
Now suppose (\ref{eq:omk-betak}) holds for some $k\ge 0$. Then
\begin{equation}\label{eq:om-induction}
(\om^{\th, k})^\f=0\al(\beta^\th_k)=0(\al^{\th,k})^\f.
\end{equation}
Note that the word $\om^{\th, k}$ begins with a $0$ and the word $\al^{\th,k}$ ends with a $0$. Furthermore, the two words $\om^{\th,k}$ and $\al^{\th,k}$ have the same length $2^k m$. Then by (\ref{eq:om-induction}) and the definitions of $(\om^{\th, i})$, $(\al^{\th,i})$ it follows that
\[
(\om^{\th,k+1})^\f=(\om^{\th,k}\overline{\om^{\th,k}})^\f=0((\al^{\th,k})^+\overline{(\al^{\th,k})^+})^\f=0(\al^{\th,k+1})^\f=0\al(\beta^\th_{k+1}).
\]
This implies that (\ref{eq:omk-betak}) also holds for $k+1$.
By induction this proves  the claim. Hence, by (\ref{eq:omk-betak}) we conclude that
\begin{equation}\label{e22}
\pi_\beta((\om^{\th,k})^\f)=\pi_\beta(0\al(\beta^\th_k))\quad\textrm{for all}\quad k\ge 0.
\end{equation}
Notice that $\beta_k^\th\nearrow \beta_*^{ \th}$ as $k\ra\f$. Thus, letting $k\ra\f$ in (\ref{e22}) and  by Lemma \ref{lem:alpha-beta} it follows that
\begin{equation}\label{e23}
\pi_\beta(\om^{\th,TM})=\pi_\beta(0\al(\beta_*^{ \th})).
\end{equation}

Note by Lemma \ref{lem:alpha-beta} that the map $q\mapsto \al(q)$ is strictly increasing and left continuous. This implies that the following map
\[
\Phi_\beta:\quad \left[\frac{1+\sqrt{5}}{2}, \beta\right)~\longrightarrow~\left[\frac{1}{\beta^2-1}, \frac{1}{\beta}\right); \quad q~\mapsto~\pi_\beta(0\al(q))
\]
is also strictly increasing and left continuous.  Indeed, for any $p, q\in[\frac{1+\sqrt{5}}{2}, \beta)$ with $p<q$, by Lemma \ref{lem:alpha-beta} it follows that
\[
\si^n(0\al(p))\preccurlyeq \al(p)\prec\al(\beta),\quad  \si^n(0\al(q))\preccurlyeq\al(q)\prec \al(\beta)
\]
for all $n\ge 0$. This implies that $0\al(p)$ and $0\al(q)$ are the lexicographically largest (\emph{greedy}) $\beta$-expansions of $\pi_\beta(0\al(p))$ and $\pi_\beta(0\al(q))$ respectively (cf.~\cite{Parry}). In \cite{Parry} it is also shown that $\pi_{\beta}$ preserves the lexicographic ordering on $\{0,1\}^{\mathbb{N}}$ when restricted to the set of greedy $\beta$-expansions. Therefore, since $0\al(p)\prec 0\al(q)$ by Lemma \ref{lem:alpha-beta}, it follows that     $\Phi_\beta(p)=\pi_\beta(0\al(p))<\pi_\beta(0\al(q))=\Phi_\beta(q)$.

Therefore,  by (\ref{e21}), (\ref{e22}) and (\ref{e23}) it follows that
\[
\pi_\beta((\om^{\th,0})^\f)<\pi_\beta((\om^{\th,1})^\f)<\cdots<\pi_\beta(\om^{\th,TM}).
\]
This implies that  $I_{\om^{\th}}=[\pi_\beta((\om^{\th,0})^\f), \pi_\beta(\om^{\th,TM})]$ is a Thue-Morse interval.
 Furthermore, by (\ref{e22}), (\ref{e23}) and the monotonicity of $\Phi_\beta$ it follows that
 \begin{equation}\label{eq:kk1}
 I_{\om^\th}=[\Phi_\beta(\beta_0^\th), \Phi_\beta(\beta_*^{ \th})] \quad\textrm{for any}\quad C_{\al^\th}=[\beta_0^\th, \beta_*^{ \th}).
 \end{equation}
 Note by (\ref{eq:intersectU}) that the closed intervals $\set{[\beta_0^\th,\beta_*^{ \th}]}_{\th\in\Theta}$ are pairwise disjoint. By  (\ref{eq:kk1}) and  the monotonicity of $\Phi_\beta$ it follows that the Thue-Morse intervals $\set{I_{\om^\th}}_{\th\in\Theta}$ are also pairwise disjoint.   This proves statement (1).

 In order to prove (2) we need the following inclusion:
 \begin{equation}\label{eq:thue-morse-intervals}
 \left[\frac{1}{\beta^2-1}, \frac{1}{\beta}\right)\setminus\bigcup_{\th\in\Theta} I_{\om^\th}\subseteq\Phi_\beta\left(\Big[\frac{1+\sqrt{5}}{2}, \beta\Big)\setminus\bigcup_{\th\in\Theta}\overline{C_{\al^\th}}\right).
 \end{equation}
Note that the function $\Phi_\beta$ is left-continuous. Unfortunately  $\Phi_\beta$ is not in general right-continuous, however it is continuous at any point of $\u$ (cf.~\cite{KomLor1}).  For this reason  we consider the following continuous function $\Psi_\beta: [\frac{1+\sqrt{5}}{2}, \beta)\ra[\frac{1}{\beta^2-1}, \frac{1}{\beta})$ which coincides with $\Phi_\beta$ on $\u$ and is affine on each closed interval $\overline{C_{\al^\th}}=[\beta_0^\th, \beta_*^{ \th}]$. To be more precise,
 \begin{equation}\label{eq:phi-0}
 \Psi_\beta(q)=\Phi_\beta(q)\quad\textrm{for any}\quad q\in\left[\frac{1+\sqrt{5}}{2}, \beta\right)\cap \u,
 \end{equation}
  and
\begin{equation}\label{eq:phi-1}
\Psi_\beta(q)=\frac{\Phi_\beta(\beta_*^{ \th})-\Phi_\beta(\beta_0^\th)}{\beta_*^{ \th}-\beta_0^\th}(q-\beta_0^\th)+\Phi_\beta(\beta_0^{\th})\quad\textrm{for any }q\in [\beta^\th_0, \beta_*^{ \th}]
\end{equation}
if $\beta\notin (\beta_0^\th,\beta_*^{ \th}]$,
and
\begin{equation}\label{eq:phi-2}
\Psi_\beta(q)=\frac{\Phi_\beta(\beta-0)-\Phi_\beta(\beta_0^\th)}{\beta-\beta_0^\th}(q-\beta_0^\th)+\Phi_\beta(\beta_0^\th)\quad\textrm{for any }q\in[\beta_0^\th, \beta)
\end{equation}
if $\beta\in (\beta^\th_0, \beta_*^{ \th}]$. Here $\Phi_\beta(\beta-0):=\lim_{q\nearrow \beta}\Phi_\beta(q)=\frac{1}{\beta}$ by the left-continuity of $\Phi_\beta$.

We claim that $\Psi_\beta$ is continuous and strictly increasing on the interval $[\frac{1+\sqrt{5}}{2}, \beta)$.
 Clearly, by (\ref{eq:phi-0})--(\ref{eq:phi-2}) and  the monotonicity of  $\Phi_\beta$ it follows  that $\Psi_\beta$ is strictly increasing. As for the continuity of $\Psi_\beta$ we consider the following four cases.
\begin{itemize}
\item[I.]  $q\in[\frac{1+\sqrt{5}}{2}, \beta)\cap\bigcup_{\th\in\Theta}(\beta^\th_0, \beta_*^{ \th})$.  Then by (\ref{eq:phi-1}) and (\ref{eq:phi-2}) it follows that $\Psi_\beta$ is continuous at  $q$.

\item[II.]  $q\in[\frac{1+\sqrt{5}}{2}, \beta)\setminus\bigcup_{\th\in\Theta}[\beta_0^\th, \beta_*^{ \th}]$. Then by (\ref{eq:intersectU}) it follows that $q\in\u$. Furthermore, there exists a sequence  $(\beta_0^{\th_j})_{j=1}^\f$ with each $\beta_0^{\th_j}$ the left endpoint of a connected component $C_{\al^{\th_j}}$  such that $\beta_0^{\th_j}\nearrow q$ as $j\ra\f$. By \eqref{eq:phi-1}, \eqref{eq:phi-2}  and the continuity of $\Phi_\beta$ in $\u$ we obtain
\[
\lim_{j\ra\f}\Psi_\beta(\beta_0^{\th_j})=\lim_{j\ra\f}\Phi_\beta(\beta_0^{\th_j})=\Phi_\beta(q)=\Psi_\beta(q).
\]
Since $\Psi_\beta$ is strictly increasing, this implies that $\Psi_\beta$ is left-continuous at $q$. Similarly, we could also find a sequence $(\beta_0^{\tilde\th_k})$ such that $\beta_0^{\tilde\th_k}\searrow q$ as $k\ra\f$. By a similar argument we conclude that $\Psi_\beta$ is also right-continuous at $q$.

\item[III.] $q=\beta_0^\th\in [\frac{1+\sqrt{5}}{2}, \beta)$.  By (\ref{eq:phi-1}) and (\ref{eq:phi-2}) it follows that $\Psi_\beta$ is right-continuous  at $q$. Furthermore, by (\ref{eq:intersectU}) there exists a sequence $(\beta_0^{\th_j})_{j=1}^\f$ such that $\beta_0^{\th_j}\nearrow q$ as $j\ra\f$. By a similar argument as in Case II we conclude  that $\Psi_\beta$ is also left-continuous  at $q$.

\item[IV] $q=\beta_{*}^{\th}\in[\frac{1+\sqrt{5}}{2}, \beta)$. By (\ref{eq:phi-1}) and (\ref{eq:phi-2}) it follows that $\Psi_\beta$ is left-continuous  at $q$. Furthermore, by (\ref{eq:intersectU}) there exists a sequence $(\beta_0^{\tilde\th_k})_{k=1}^\f$ such that $\beta_0^{\tilde\th_k}\searrow q$ as $k\ra\f$. By a similar argument as in Case II we could prove  that $\Psi_\beta$ is also right-continuous  at $q$.
\end{itemize}
Note by (\ref{eq:phi-1}) and (\ref{eq:phi-2}) that $\Psi_\beta(\frac{1+\sqrt{5}}{2})=\Phi_\beta(\frac{1+\sqrt{5}}{2})=\frac{1}{\beta^2-1}$ and $\lim_{q\nearrow\beta}\Psi_\beta(q)= \Phi_\beta(\beta-0)=\frac{1}{\beta}$. Therefore, by the monotonicity and continuity  of $\Psi_\beta$ it follows that
 \begin{equation}\label{eq:kk2}
 \Psi_\beta\left(\Big[\frac{1+\sqrt{5}}{2}, \beta\Big)\right)=\left[\frac{1}{\beta^2-1}, \frac{1}{\beta}\right).
 \end{equation}
 Furthermore,  by (\ref{eq:kk1})  and (\ref{eq:phi-1}) it follows that  if $\beta\notin(\beta_0^\th, \beta_*^{ \th}]$ then the interval $[\beta_0^\th, \beta_*^{ \th}]$ and the Thue-Morse interval $I_{\om^\th}$ satisfy
   \begin{equation}\label{eq:kon2}
 I_{\om^\th}=[\Phi_{\beta}(\beta_0^\th), \Phi_\beta(\beta_*^{ \th})]=[\Psi_\beta(\beta_0^\th), \Psi_\beta(\beta_*^{ \th})]=\Psi_\beta([\beta_0^\th,\beta_*^{ \th}]).
 \end{equation}
 Similarly, by (\ref{eq:kk1}) and (\ref{eq:phi-2}) it follows that if $\beta\in(\beta_0^\th, \beta_*^{ \th}]$, then the interval $[\beta_0^\th, \beta)$ and the truncated Thue-Morse interval $I_{\om^\th}\cap[\frac{1}{\beta^2-1}, \frac{1}{\beta})$ satisfy
 \begin{equation}\label{eq:kon3}
 I_{\om^\th}\cap\left[\frac{1}{\beta^2-1}, \frac{1}{\beta}\right)=\Psi_\beta([\beta_0^\th, \beta)).
 \end{equation}
 Therefore, by (\ref{eq:phi-0}) and (\ref{eq:kk2})--(\ref{eq:kon3}) it follows that
 \begin{align*}
 \left[\frac{1}{\beta^2-1}, \frac{1}{\beta}\right)\setminus\bigcup_{\th\in\Theta}I_{\om^\th}&=\Psi_\beta\left(\Big[\frac{1+\sqrt{5}}{2}, \beta\Big)\right)\setminus\bigcup_{\th\in\Theta} \Psi_\beta([\beta_0^\th, \beta_*^{ \th}])\\
 &\subseteq\Psi_\beta\left(\Big[\frac{1+\sqrt{5}}{2}, \beta\Big)\setminus\bigcup_{\th\in\Theta}[\beta_0^\th, \beta_*^{ \th}]\right)\\
 &=\Phi_\beta\left(\Big[\frac{1+\sqrt{5}}{2}, \beta\Big)\setminus\bigcup_{\th\in\Theta}\overline{C_{\al^\th}}\right).
 \end{align*}
 This proves  (\ref{eq:thue-morse-intervals}).

Hence, by (\ref{eq:intersectU}) and (\ref{eq:thue-morse-intervals}) it follows that
\begin{align*}
\left[\frac{1}{\beta^2-1}, \frac{1}{\beta}\right)\setminus\bigcup_{\th\in\Theta} I_{\om^\th}&\subseteq\Phi_\beta\left(\Big[\frac{1+\sqrt{5}}{2}, \beta\Big)\setminus\bigcup_{\th\in\Theta}\overline{C_{\al^\th}}\right)\\
&\subseteq\Phi_\beta\left(\Big[\frac{1+\sqrt{5}}{2}, \beta\Big)\cap\u\right)\;\subseteq U_\beta,
\end{align*}
where the last inclusion holds by the following observation. Note that  for any $q\in[\frac{1+\sqrt{5}}{2}, \beta)\cap\mathcal U$ the quasi-greedy expansion $\al(q)\in \widetilde U_q$. Since $\al(q)\prec \al(\beta)$,  by Lemma \ref{lem:univoque} it follows that
$0\al(q)\in\widetilde U_\beta$, and hence $\Phi_\beta(q)\in U_\beta$.
This proves statement (2).

Observe by symmetry that
\[
\left(\frac{1}{\beta(\beta-1)}, \frac{\beta}{\beta^2-1}\right]\setminus\bigcup_{\th\in\Theta} J_{\om^\th}=\frac{1}{\beta-1}-\left(\left[\frac{1}{\beta^2-1}, \frac{1}{\beta}\right)\setminus\bigcup_{\th\in\Theta} I_{\om^\th}\right),
\]
and $\frac{1}{\beta-1}-U_\beta=U_\beta$. Therefore statement (3) follows from statement (2). Note that each word $\om^\th$ corresponds to a unique connected component $C_{\al^\th}=[\beta^\th_0, \beta_*^{ \th})$. By \eqref{eq:omk-betak} we have $(\om^\th)^\f=0\al(\beta^\th_0)$. For the first connected component $[\frac{1+\sqrt{5}}{2}, \beta_{KL})$ we have $\al(\frac{1+\sqrt{5}}{2})=(10)^\f$, and then statement (4) holds in this case since $(\om^\th)=(01)^\f$. For the other connected components $C_{\al^\th}=[\beta_0^\th, \beta_*^{ \th}]$ we have $\beta^\th_0\in[\frac{1+\sqrt{5}}{2}, \beta)\cap\overline{\u}$. Hence, by using $\al(\beta^\th_0)\prec\al(\beta)$ in  Lemma \ref{lem:univoque} we conclude that (cf.~\cite{KomLor1})
\[(\om^\th)^\f=0\al(\beta^\th_0)\in \widetilde U_{\beta}.\]
So statement (4) follows from Lemma \ref{unique}. Finally statement (5) follows from the proof of  \cite[Lemma 2.3]{JSS}. It is a consequence of the proof of this lemma that every $\om^\th$ is the  concatenation of words from the sets
 $$\{1(10)^j0: j=0,1,\ldots\}\quad\textrm{and}\quad \{0(01)^j1: j=0,1,\ldots\}.$$
 Consequently we can take the constant $C=2$.
\end{proof}
We emphasise that the $C$ appearing in property $(5)$ from Proposition \ref{covering prop} is a uniform bound over all $\theta\in \Theta$ and $n$. Note it follows from the construction of the Thue-Morse chain that every word $\omega^{\theta,k}$ appearing in the Thue-Morse chain $(\omega^{\theta,k})$ also satisfies$$\frac{\#\{1\leq i \leq |\omega^{\theta,k}|:\omega_i^{\theta,k}=0\}}{|\omega^{\theta,k}|}=\frac{1}{2}$$ and $$\Big|\#\{1\leq i \leq n: \omega_i^{\theta,k}=0\}-\#\{1\leq i \leq n: \omega_i^{\theta,k}=1\}\Big|\leq C$$ for all  $\theta\in \Theta$ and $1\leq n\leq |w^{\theta,k}|.$ Moreover, it is a straightforward exercise to show that
$$\lim_{n\to\infty}\frac{\#\{1\leq i \leq n:\omega_{i}^{\th,TM}=0\}}{n}=\frac{1}{2}.$$

We also highlight the following equalities. To a finite sequence $\omega=(\omega_1,\ldots,\omega_n)\in\{0,1\}^n$ we associate the concatenation of maps $T_{\omega}:=(T_{\omega_n}\circ \cdots \circ T_{\omega_1}).$ The following holds for any $\beta\in(1,2)$ and Thue-Morse chain $(\omega^k)$:
\begin{equation}
\label{fixed point}
T_{\omega^{k}}(\pi_{\beta}((\omega^k)^{\infty}))=\pi_{\beta}((\omega^k)^{\infty})\quad \textrm{and}\quad  T_{\overline{\omega^{k}}}(\pi_{\beta}((\overline{\omega^k})^{\infty}))=\pi_{\beta}((\overline{\omega^k})^{\infty})
\end{equation}for all $\omega^k.$ Moreover using $\om^{k+1}=\om^k\overline{\om^k}$ it follows that
\begin{equation}
\label{level up1}
T_{\omega^k}(\pi_{\beta}((\omega^{k+1})^{\infty}))=\pi_{\beta}((\overline{\omega^{k+1}})^{\infty}),
\end{equation}
and
\begin{equation}
\label{level up2}
T_{\overline{\omega^{k}}}(\pi_{\beta}((\overline{\omega^{k+1}})^{\infty}))=\pi_{\beta}((\omega^{k+1})^{\infty})
\end{equation}for all $\omega^k.$

\section{Proof of Theorem \ref{Main theorem}}
\label{proofs}
Equipped with the preliminaries detailed in Section \ref{Preliminaries}, we are now in a position to prove Theorem \ref{Main theorem}. We split our proof into two parameter spaces. Note by Theorem \ref{frequency theorem} it suffices to consider the interval $[\beta_{KL},\beta_T]$. First we examine the case where $\beta\in[\beta_{KL},\beta_{T})$ before moving on to the specific case where $\beta=\beta_{T}.$ Our proof in either case involves splitting $S_{\beta}$ into a left interval, a centre interval, and a right interval (see Figure \ref{Fig2} for $\beta\in[\beta_{KL}, \beta_T)$ and Figure \ref{Fig3} for $\beta=\beta_T$). Loosely speaking, in our proofs we will see that if a point is contained in the left interval or the right interval, then there is a specific sequence of transformations that map our point back into $S_{\beta},$ where importantly the frequency of $T_0$'s within these maps is approximately $1/2$. If a point is contained in the centre interval then we have a choice between a sequence of maps that increase the frequency of $T_0$'s and map our point back to $S_{\beta}$, or a sequence of maps that decrease the frequency of $T_0$'s  and map our point back to $S_{\beta}$. Importantly, in this case we will have strong bounds on how much the frequency can change. In each case we return to $S_{\beta}$. By carefully choosing which maps we perform we can construct the desired simply normal expansion.


\begin{proof}[Proof of Theorem \ref{Main theorem} for $\beta\in[\beta_{KL},\beta_{T})$]

Fix $\beta\in [\beta_{KL},\beta_{T}).$ Let us start by making several observations. First of all, by Lemma \ref{unique} it suffices to prove that every $x\in S_\beta$ has a simply normal $\beta$-expansion.  What is more, it is a consequence of Lemma \ref{unique} that one may assume that there exists no $a\in \{T_{0},T_1\}^*$ such that
\begin{equation}
\label{non unique}
a(x)\in U_{\beta}\setminus\Big\{0,\frac{1}{\beta-1}\Big\}.
\end{equation}
 Similarly, adopting the notation used in Proposition \ref{covering prop}, one may assume that there exists no $a\in \{T_0,T_1\}^*$ such that
\begin{equation}
\label{TM misses}
a(x)\in \bigcup_{\theta\in\Theta}\bigcup_{k=0}^\f\{\pi_{\beta}((\omega^{\theta,k})^{\infty}),\pi_{\beta}((\overline{\omega^{\theta,k}})^{\infty})\}\cup\bigcup_{\theta\in\Theta}\{\pi_{\beta}(\omega^{\theta,TM}), \pi_{\beta}(\overline{\omega^{\theta,TM}})\}.
\end{equation}

 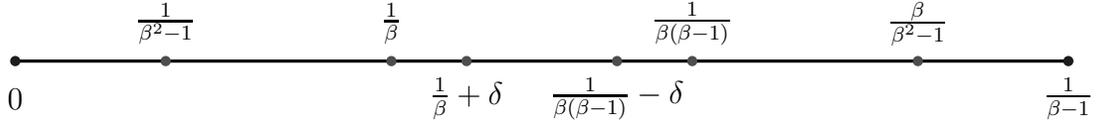
\begin{figure}[h!]
\begin{center}
\begin{tikzpicture}[
    scale=10,
    axis/.style={very thick},
    important line/.style={thick},
    dashed line/.style={dashed, thin},
    pile/.style={thick, ->, >=stealth', shorten <=2pt, shorten
    >=2pt},
    every node/.style={color=black}
    ]
    \draw[axis] (0,0)  -- (1.4,0) node(xline)[right]{};

   \draw  (0, 0) node[blue,fill,circle,scale=0.35]{};
      \node[] at (0, -0.05){$0$};

                       \draw  (0.2, 0) node[red,  fill,circle,scale=0.35]{};
      \node[] at (0.2, 0.05){$\frac{1}{\beta^2-1}$};

                     \draw  (0.5, 0) node[red,  fill,circle,scale=0.35]{};
      \node[] at (0.5, 0.05){$\frac{1}{\beta}$};

       \draw  (0.6, 0) node[red,  fill,circle,scale=0.35]{};
      \node[] at (0.6, -0.05){$\frac{1}{\beta}+\delta$};

         \draw  (0.8, 0) node[red,  fill,circle,scale=0.35]{};
      \node[] at (0.8, -0.05){$\frac{1}{\beta(\beta-1)}-\delta$};

      \draw  (0.9, 0) node[red,  fill,circle,scale=0.35]{};
      \node[] at (0.9, 0.05){$\frac{1}{\beta(\beta-1)}$};

          \draw  (1.2, 0) node[red,  fill,circle,scale=0.35]{};
      \node[] at (1.2, 0.05){$\frac{\beta}{\beta^2-1}$};

     \draw  (1.4, 0) node[blue, fill,circle,scale=0.35]{};
      \node[] at (1.4, -0.05){$\frac{1}{\beta-1}$};

\end{tikzpicture}
\end{center}
\caption{The attractor $O_\beta=[\frac{1}{\beta^2-1}, \frac{\beta}{\beta^2-1}]$. The switch region $S_\beta=[\frac{1}{\beta}, \frac{1}{\beta(\beta-1)}]$ is partitioned into three subintervals by the two points $\frac{1}{\beta}+\delta$ and $\frac{1}{\beta(\beta-1)}-\delta$.}\label{Fig2}
\end{figure}

It is a consequence of $\beta\in[\beta_{KL},\beta_{T})$ that $$(T_{1}\circ T_{0})\Big (\frac{1}{\beta}\Big)\in\Big(\frac{1}{\beta(\beta-1)},\frac{\beta}{\beta^2-1}\Big) \textrm{ and }(T_0\circ T_1)\Big(\frac{1}{\beta(\beta-1)}\Big)\in\Big(\frac{1}{\beta^2-1},\frac{1}{\beta}\Big).$$ Therefore, there exists $\delta(\beta):=\delta>0$ such that if
\begin{equation}
\label{right jump}
x\in \Big[\frac{1}{\beta},\frac{1}{\beta}+\delta\Big) \textrm{ then } (T_1\circ T_0)(x)\in\Big(\frac{1}{\beta(\beta-1)},\frac{\beta}{\beta^2-1}\Big),
\end{equation}
and if
\begin{equation}
\label{left jump}
x\in\Big(\frac{1}{\beta(\beta-1)}-\delta,\frac{1}{\beta(\beta-1)}\Big] \textrm{ then } (T_0\circ T_1)(x)\in\Big(\frac{1}{\beta^2-1},\frac{1}{\beta}\Big).
\end{equation}
We also recall from \cite{Bak} that there exists a parameter $K:=K(\beta)\in\mathbb{N}$ such that if
\begin{equation}
\label{balanced return right}
x\in \Big[\frac{1}{\beta}+\delta,\frac{1}{\beta(\beta-1)}-\delta\Big]\textrm{ then }(T_1^j\circ T_{0})(x)\in O_{\beta}
\end{equation}for some $1\leq j\leq K$. Similarly, for the same parameter $K$ if
\begin{equation}
\label{balanced return left}
x\in \Big[\frac{1}{\beta}+\delta,\frac{1}{\beta(\beta-1)}-\delta\Big] \textrm{ then }(T_0^j\circ T_{1})(x)\in  O_{\beta}
\end{equation} for some $1\leq j\leq K$. The existence of the $K$ appearing in \eqref{balanced return right} and \eqref{balanced return left} is essentially a consequence of the fact that $T_0$ and $T_1$ scale distances between arbitrary points and their unique fixed points by a factor $\beta$.

Equipped with the above observations we now fix an $x\in S_\beta$ and describe an algorithm which constructs an element of $\Omega_{\beta}(x)$ that corresponds to a simply normal expansion via Lemma \ref{Bijection lemma}. As mentioned above it is useful to partition $S_{\beta}$ into three intervals (see Figure \ref{Fig2}). This we do now.

\textbf{Case 1.} If $x\in [\frac{1}{\beta},\frac{1}{\beta}+\delta)$ then $$(T_{1}\circ T_{0})(x)\in  \Big(\frac{1}{\beta(\beta-1)},\frac{\beta}{\beta^2-1}\Big)$$by \eqref{right jump}. By our assumptions we know that $(T_{1}\circ T_{0})(x)\notin U_{\beta}$. Therefore by Proposition \ref{covering prop} we have $(T_{1}\circ T_{0})(x)\in J_{\omega^{\theta_1}}$ for some $\theta_1\in \Theta.$ By \eqref{TM misses} we know that $$\pi_{\beta}((\overline{\omega^{\theta_1,k_1+1}})^{\infty})<(T_{1}\circ T_{0})(x)<\pi_{\beta}((\overline{\omega^{\theta_1,k_1}})^{\infty})$$ for some $k_1\geq 0$. By \eqref{fixed point} we know that $\pi_{\beta}((\overline{\omega^{\theta_1,k_1}})^{\infty})$ is the unique fixed point of the map $T_{\overline{\omega^{\theta_1,k_1}}}.$ Importantly this map expands distances by a factor $\beta^{|\overline{\omega^{\theta_1,k_1}}|}.$ Therefore it follows from the monotonicity of our maps and \eqref{fixed point}--\eqref{level up2} that there must exist $n_1\in\mathbb{N}$ such that
$$T_{\overline{\omega^{\theta_1,k_1}}}^{n_1}\big((T_{1}\circ T_0)(x)\big)\in [\pi_{\beta}((\omega^{\th_1,k_1+1})^{\infty}),\pi_{\beta}((\overline{\omega^{\th_1,k_1+1}})^{\infty})].$$ Here $T_{\overline{\om^{\th_1, k_1}}}^{n_1}$ stands for the $n_1$ times composition of the map $T_{\overline{\om^{\th_1,k_1}}}$.
In the above inclusion it is not important that this image point is contained in this particular interval parameterized by $\omega^{\theta_1,k_1+1}$. What is important is that it is contained in $O_{\beta}$. This means we can reuse Proposition \ref{covering prop}.

At this point in our algorithm we stop and consider where $$T_{\overline{\omega^{\theta_1,k_1}}}^{n_1}\big((T_{1}\circ T_0)(x)\big)$$ lies within $O_{\beta}$. If it is contained in $S_{\beta}$ we stop and let $$a^1:=(T_0,T_1,(T_{\overline{\omega^{\theta_1,k_1}}})^{n_1}).$$
 If this image is not contained in $S_{\beta},$ then we know by \eqref{non unique} and Proposition \ref{covering prop} that it must be contained in a Thue-Morse interval. In which case, repeating the above argument, there must exist $\theta_2,k_2,$ and $n_2$ such that $$T_{\omega^{\theta_2,k_2}}^{n_2}\big(T_{\overline{\omega^{\theta_1,k_1}}}^{n_1}((T_{1}\circ T_0)(x))\big)\in O_{\beta}\quad\textrm{or}\quad T_{\overline{\omega^{\theta_2,k_2}}}^{n_2}\big(T_{\overline{\omega^{\theta_1,k_1}}}^{n_1}((T_{1}\circ T_0)(x))\big)\in O_{\beta}.$$ If this image point is in $S_{\beta}$ we stop and let $$a^1:=(T_0,T_1,(T_{\overline{\omega^{\theta_1,k_1}}})^{n_1},(T_{\omega^{\theta_2,k_2}})^{n_2})\quad \textrm{or}\quad a^1:=(T_0,T_1,(T_{\overline{\omega^{\theta_1,k_1}}})^{n_1},(T_{\overline{\omega^{\theta_2,k_2}}})^{n_2})$$ accordingly. We can repeat this process indefinitely. If we are never mapped into the switch region then it follows from Proposition \ref{covering prop} properties $(4)$ and $(5)$ that we've constructed an element of $\Omega_{\beta}(x)$ with limiting frequency of zeros $1/2$. Which by Lemma \ref{Bijection lemma} proves our result. Alternatively, if this process eventually maps $x$ into $S_{\beta},$ then the corresponding sequence $a^1\in \{T_0,T_1\}^*$ satisfies $a^1(x)\in S_{\beta}$ and has the following useful properties as a consequence of Proposition \ref{covering prop}:
$$|a^1|_1=|a^1|_0$$ and $$\Big|\#\{1\leq i \leq n: a^1_{i}=T_0\}- \#\{1\leq i \leq n: a_{i}^1=T_1\}\Big|\leq C$$ for all $1\leq n \leq |a^1|.$

\textbf{Case 2.} The case where $x\in (\frac{1}{\beta(\beta-1)}-\delta, \frac{1}{\beta(\beta-1)}]$ is handled in the same way as Case $1$. The difference being in this case, instead of intially applying the map $T_1\circ T_0$ we apply $T_0\circ T_1$. Our orbit then travels through successive Thue-Morse intervals before landing in the switch region $S_{\beta}$, or our image never maps into $S_{\beta}$ and then  we have immediately constructed a simply normal expansion. In the first case we construct a sequence $a^1\in \{T_0,T_1\}^*$ which satisfies $a^1(x)\in S_{\beta},$
$$|a^1|_1=|a^1|_0$$ and $$\Big|\#\{1\leq i \leq n: a^1_{i}=T_0\}- \#\{1\leq i \leq n: a^1_{i}=T_1\}\Big|\leq C$$ for all $1\leq n \leq |a^1|.$

\textbf{Case 3.} When $x\in [\frac{1}{\beta}+\delta,\frac{1}{\beta(\beta-1)}-\delta]$ we can initially apply $T_0$ or $T_1$. By (\ref{balanced return right}) and (\ref{balanced return left}) we then successively apply either $T_1$ or $T_0$ until $(T_1^j\circ T_0)(x)\in O_{\beta}$ or $(T_0^j\circ T_1)(x)\in O_{\beta}.$ Once $x$ is mapped into $O_{\beta}$ we then proceed as in Case $1$. We travel through successive Thue-Morse intervals before being eventually mapped into $S_{\beta},$ or $x$ is never mapped into $S_{\beta}$ and we then automatically have a simply normal expansion. In the case where we initially apply $T_0$, by (\ref{balanced return right}) we will have constructed a sequence $a^1\in\{T_0,T_1\}^*$ that satisfies $a^1(x)\in S_{\beta}$,
\begin{equation}
\label{positive contribution}
0\leq |a^1|_1-|a^1|_0\leq K,
\end{equation} and  $$\Big|\#\{1\leq i \leq n: a^1_{i}=T_0\}- \#\{1\leq i \leq n: a^1_{i}=T_1\}\Big|\leq C+K$$ for all $1\leq n \leq |a^1|.$ If we initially applied $T_1,$ then by (\ref{balanced return left}) we will have constructed a sequence $a^1\in\{T_0,T_1\}^*$ that satisfies $a^1(x)\in S_{\beta}$,
\begin{equation}
\label{negative contribution}
-K\leq |a^1|_1-|a^1|_0\leq 0,
\end{equation} and $$\Big|\#\{1\leq i \leq n: a^1_{i}=T_0\}- \#\{1\leq i \leq n: a^1_{i}=T_1\}\Big|\leq C+K$$ for all $1\leq n \leq |a^1|.$

Now suppose we've constructed a finite sequence $a^m\in\{T_0,T_1\}^*$ such that $a^m(x)\in S_\beta,$
\begin{equation}
\label{balanced1}
\big||a^m|_1-|a^m|_0\big|\leq K,
\end{equation}
and
\begin{equation}
\label{balanced2}
\Big|\#\{1\leq i \leq n: a^m_{i}=T_0\}- \#\{1\leq i \leq n: a^m_{i}=T_1\}\Big|\leq C+K
\end{equation} for all $1\leq n \leq |a^m|.$ We now construct a sequence $a^{m+1}$ that has $a^m$ as a prefix and satisfies \eqref{balanced1} and \eqref{balanced2}. If $a^m(x)\in [\frac{1}{\beta},\frac{1}{\beta}+\delta)$ or $a^m(x)\in(\frac{1}{\beta(\beta-1)}-\delta,\frac{1}{\beta(\beta-1)}]$ then we repeat the arguments as in Case $1$ or Case $2$ respectively. In either case we construct a sequence of transformations $a^{m+1}\in\{T_0,T_1\}^*$ that begins with $a^m$ and satisfies $a^{m+1}(x)\in S_{\beta}$, $$\big||a^{m+1}|_1-|a^{m+1}|_0\big|\leq K,$$ and $$\Big|\#\{1\leq i \leq n: a^{m+1}_{i}=T_0\}- \#\{1\leq i \leq n: a^{m+1}_i=T_1\}\Big|\leq C+K$$ for all $1\leq n \leq |a^{m+1}|.$ If $a^m(x)\in [\frac{1}{\beta}+\delta,\frac{1}{\beta(\beta-1)}-\delta]$ then we consider the sign of $|a^{m}|_1-|a^m|_0.$ If $0\leq |a^{m}|_1-|a^m|_0 \leq K$ then we repeat the arguments given in Case $3$ when we initally apply $T_1$. In this case \eqref{negative contribution} guarantees that  $$\big||a^{m+1}|_1-|a^{m+1}|_0\big|\leq K.$$ We also have $a^{m+1}(x)\in S_{\beta}$ and $$\Big|\#\{1\leq i \leq n: a^{m+1}_i=T_0\}- \#\{1\leq i \leq n: a^{m+1}_{i}=T_1\}\Big|\leq C+K$$ for all $1\le n\le |a^{m+1}|$.  If $|a^{m}|_1-|a^m|_0$ is negative then we repeat the above argument except we use Case $3$ where we first apply $T_0$.

Clearly we can repeat the above steps indefinitely. In doing so we construct an infinite sequence  in $\Omega_{\beta}(x)$. It is a consequence of \eqref{balanced2} that this sequence has the desired frequency. Therefore by Lemma \ref{Bijection lemma} we know that $x$ has a simply normal expansion.
\end{proof}

\bigskip

\begin{proof}[Proof of Theorem \ref{Main theorem} for $\beta=\beta_T$]
We start with an  observation. For any $J\in\mathbb{N}$ there exists $\delta_J>0$ such that if $x\in [\frac{1}{\beta_T},\frac{1}{\beta_T}+\delta_J)$ then
\begin{equation}
\label{plus one}
((T_{1}\circ T_{0})^J \circ T_{1}^2\circ T_{0})(x)\in O_{\beta_T}.
\end{equation}  This is because $(T_1^2\circ T_0)(\frac{1}{\beta_T})=\frac{1}{\beta_T^2-1}$. Similarly, if $x\in (\frac{1}{\beta_T(\beta_T-1)}-\delta_J,\frac{1}{\beta_T(\beta_T-1)}]$ then
\begin{equation}
\label{minus one}
((T_{0}\circ T_{1})^J \circ T_{0}^2\circ T_{1})(x)\in O_{\beta_T}.
\end{equation}
 \begin{figure}[h!]
\begin{center}
\begin{tikzpicture}[
    scale=10,
    axis/.style={very thick},
    important line/.style={thick},
    dashed line/.style={dashed, thin},
    pile/.style={thick, ->, >=stealth', shorten <=2pt, shorten
    >=2pt},
    every node/.style={color=black}
    ]
    \draw[axis] (0,0)  -- (1.4,0) node(xline)[right]{};

   \draw  (0, 0) node[blue,fill,circle,scale=0.35]{};
      \node[] at (0, -0.05){$0$};

                       \draw  (0.2, 0) node[red,  fill,circle,scale=0.35]{};
      \node[] at (0.2, 0.05){$\frac{1}{\beta_T^2-1}$};

                     \draw  (0.45, 0) node[red,  fill,circle,scale=0.35]{};
      \node[] at (0.45, 0.05){$\frac{1}{\beta_T}$};

       \draw  (0.55, 0) node[red,  fill,circle,scale=0.35]{};
      \node[] at (0.55, -0.05){$\frac{1}{\beta_T}+\delta_J$};

         \draw  (0.85, 0) node[red,  fill,circle,scale=0.35]{};
      \node[] at (0.85, -0.05){$\frac{1}{\beta_T(\beta_T-1)}-\delta_J$};

      \draw  (0.95, 0) node[red,  fill,circle,scale=0.35]{};
      \node[] at (0.95, 0.05){$\frac{1}{\beta_T(\beta_T-1)}$};

          \draw  (1.2, 0) node[red,  fill,circle,scale=0.35]{};
      \node[] at (1.2, 0.05){$\frac{\beta_T}{\beta_T^2-1}$};

     \draw  (1.4, 0) node[blue, fill,circle,scale=0.35]{};
      \node[] at (1.4, -0.05){$\frac{1}{\beta_T-1}$};

\end{tikzpicture}
\end{center}
\caption{The attractor $O_{\beta_T}=[\frac{1}{\beta_T^2-1}, \frac{\beta_T}{\beta_T^2-1}]$. The switch  $S_{\beta_T}=[\frac{1}{\beta_T}, \frac{1}{\beta_T(\beta_T-1)}]$ is partitioned into three subintervals by the two points $\frac{1}{\beta_T}+\delta_J$ and $\frac{1}{\beta_T(\beta_T-1)}-\delta_J$.}\label{Fig3}
\end{figure}
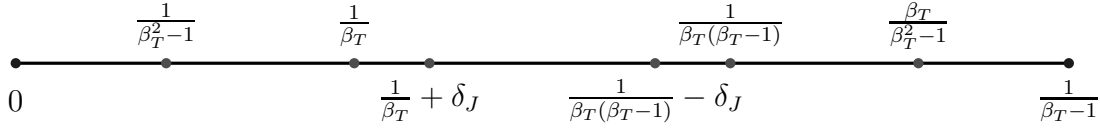

Moreover, for each $J\in\mathbb{N}$ there exists $K_{J}\in\mathbb{N}$ such that if $x\in [\frac{1}{\beta_T}+\delta_J,\frac{1}{\beta_T(\beta_T-1)}-\delta_J],$ then
\begin{equation}
\label{right return}
(T_{1}^i\circ T_0)(x)\in O_{\beta_T}
\end{equation} for some $1\leq i \leq K_J,$ and
\begin{equation}
\label{left return}
(T_{0}^i\circ T_1)(x)\in O_{\beta_T}
\end{equation} for some $1\leq i \leq K_J.$ As in our proof for $\beta\in [\beta_{KL},\beta_T)$ it is useful to partition $S_{\beta_T}$ into three intervals (see Figure \ref{Fig3}). This time however our partition will depend upon $J$.

\textbf{Case 1.} If $x\in [\frac{1}{\beta_T},\frac{1}{\beta_T}+\delta_J)$ then by \eqref{plus one} we know that $((T_{1}\circ T_{0})^J \circ T_{1}^2\circ T_{0})(x)\in O_{\beta_T}.$ Repeating arguments given in our proof for $\beta\in[\beta_{KL},\beta_T),$ we may assume that we may concatenate $(T_0,T_1,T_1,(T_0,T_1)^J)$ with a sequence of maps that map $x$ back into $S_{\beta_T},$ and satisfy Properties $(4)$ and $(5)$ of Proposition \ref{covering prop}. Letting $a\in \{T_0,T_1\}^*$ be the concatenation of $(T_0,T_1,T_1,(T_0,T_1)^J)$ with this second sequence of maps, we can assert by Proposition \ref{covering prop} and \eqref{plus one} that $a(x)\in S_{\beta_T}$ and
\begin{equation}
\label{close1}
\frac{\#\{1\leq i \leq |a|: a_{i}=T_0\}}{|a|}\in\Big[\frac{J+1}{2J+3},\frac{1}{2}\Big]
\end{equation}

\textbf{Case 2.}  If $x\in (\frac{1}{\beta_T(\beta_T-1)}-\delta_J,\frac{1}{\beta_T(\beta_T-1)}],$ then by (\ref{minus one}) and a similar analysis to that done in Case $1,$ except this time first applying the sequence of maps $(T_1,T_0,T_0,(T_1,T_0)^J),$ implies the existence of a sequence $a\in\{T_0,T_1\}^*$ such that $a(x)\in S_{\beta_T}$ and
\begin{equation}
\label{close2}
\frac{\#\{1\leq i \leq |a|: a_{i}=T_0\}}{|a|}\in\Big[\frac{1}{2},\frac{J+2}{2J+3}\Big].
\end{equation}

\textbf{Case 3.} If $x\in [\frac{1}{\beta_T}+\delta_J,\frac{1}{\beta_T(\beta_T-1)}-\delta_J]$ then by \eqref{right return} and \eqref{left return} we know that $(T_1^i\circ T_0)(x)\in O_{\beta_T}$ for some $1\leq i\leq K_J,$ and $(T_0^j\circ T_1)(x)\in O_{\beta_T}$  for some $1\leq j\leq K_J.$ Repeating the arguments given in Case $3$ of our proof for $\beta\in[\beta_{KL}, \beta_T)$ where we appealed to Proposition \ref{covering prop}, we may assert that for such an $x$ there exists a sequence $a\in\{T_0,T_1\}^*$ such that $a(x)\in S_{\beta_T}$ and $a$ satisfies
\begin{equation}
\label{push up}
0\leq |a|_1-|a|_0 \leq K_J
\end{equation} if we initially applied $T_0,$ or if we initially applied $T_1$ then $a$ satisfies
\begin{equation}
\label{push down}
-K_J\leq |a|_1-|a|_0 \leq 0.
\end{equation}

Having described the maps we can perform in each of the three subintervals of $S_{\beta_T}$, let us now fix an $x\in S_{\beta}$. Moreover, let $\varepsilon_n=n^{-1}$ and let $(J_n)$ be a strictly increasing sequence of natural numbers such that
\begin{equation}
\label{quantifiably close}
\frac{1}{2}-\varepsilon_n<\frac{J_n+1}{2J_n+3}\quad\textrm{and}\quad \frac{J_n+2}{2J_n+3}<\frac{1}{2}+\varepsilon_n
\end{equation} for all $n\ge 1$.

We now show how to construct a simply normal expansion of $x\in S_{\beta_T}$. By repeatedly applying the maps detailed in Cases $1$, $2,$ and $3,$ we can construct an arbitrarily long sequence of maps $a^1$ that satisfies $a^1(x)\in S_{\beta_T}$ and
\begin{equation}
\label{epsilon1}
\frac{\#\{1\leq i \leq |a^1|: a^1_{i}=T_0\}}{|a^1|}\in \Big(\frac{1}{2}-\varepsilon_1,\frac{1}{2}+\varepsilon_1\Big).
\end{equation}
To construct such an $a^1$ the strategy is as follows. Consider the partition of $S_{\beta_T}$ given by $J_1$. If our point is mapped into either of the intervals described in Cases $1$ and $2$ then we always perform the sequence of maps that satisfy \eqref{close1} or \eqref{close2}. If we are mapped into the interval covered by Case $3$ we have a choice. If the number  of $T_0$'s appearing in the sequence of maps we have constructed so far exceeds the number of $T_1$'s, then we apply the sequence of maps corresponding to \eqref{push up}.  If the number of $T_1$'s appearing in the sequence of maps we have constructed so far exceeds the number of $T_0$'s, then we apply the sequence of maps corresponding to \eqref{push down}. Since each of the sequences of maps described in Cases $1$, $2$, and $3$ map us back into $S_{\beta_T},$ we can clearly repeat this process indefinitely. Since the maps described by Case $3$ increase or decrease the difference between the number of $T_0$'s and $T_1$'s by at most $K_{J_1}$, it follows that any sufficiently large sequence of maps constructed using the above steps satisfies $a^1(x)\in S_{\beta_T}$ and \eqref{epsilon1} by \eqref{quantifiably close}

Now we repeat the same process but with $x$ replaced by $a^1(x)$ and $J_1$ replaced by $J_2$. We may assert that there exists $a^2$ that extends $a^1$ such that $a^2(x)\in S_{\beta_T},$ and
 $$\frac{\#\{1\leq i \leq |a^2|: a^2_{i}=T_0\}}{|a^2|}\in\Big(\frac{1}{2}-\varepsilon_2,\frac{1}{2}+\varepsilon_2\Big)$$ by \eqref{quantifiably close}. It is a consequence of property $(5)$ of Proposition \ref{covering prop}, and the fact that $a^1$ may be made arbitrarily long, that we may also assume that $a^2$ satisfies
$$\frac{\#\{1\leq i \leq n: a^2_{i}=T_0\}}{n}\in\Big(\frac{1}{2}-2\varepsilon_1,\frac{1}{2}+2\varepsilon_1\Big)$$ for all $|a^1|\leq n < |a^2|$. Importantly $a^2$ can also be made to be arbitrarily long.

Now assume that we have constructed $a^1,\ldots, a^N$ such that $a^N$ is arbitrarily long, $a^N(x)\in S_{\beta_T},$
 $$\frac{\#\{1\leq i \leq |a^N|:a^N_{i}=T_0\}}{|a^N|}\in\Big(\frac{1}{2}-\varepsilon_N,\frac{1}{2}+\varepsilon_N\Big),$$ and for all $|a^{j}|\leq n < |a^{j+1}|$  with $1\le j<N$ we have
\begin{equation}
\label{closea}
\frac{\#\{1\leq i \leq n: a^N_{i}=T_0\}}{n}\in\Big(\frac{1}{2}-2\varepsilon_{j},\frac{1}{2}+2\varepsilon_{j}\Big).
\end{equation} By repeating the above arguments, this time considering $a^N(x)$ and $J_{N+1},$ we may construct an arbitrarily long sequence $a^{N+1}$ that extends $a^N$ and satisfies
$a^{N+1}(x)\in S_{\beta_T},$ $$\frac{\#\{1\leq i \leq |a^{N+1}|: a^{N+1}_{i}=T_0\}}{|a^{N+1}|}\in\Big(\frac{1}{2}-\varepsilon_{N+1},\frac{1}{2}+\varepsilon_{N+1}\Big),$$ and  for all $|a^N|\leq n < |a^{N+1}|$ we have
\begin{equation}
\label{closeb}
\frac{\#\{1\leq i \leq n: a^{N+1}_{i}=T_0\}}{n}\in\Big(\frac{1}{2}-2\varepsilon_{N},\frac{1}{2}+2\varepsilon_{N}\Big).
\end{equation}

Continuing indefinitely we construct an element of $\Omega_{\beta_T}(x)$. This sequence corresponds to a simply normal expansion by Lemma \ref{Bijection lemma}, \eqref{closea}, and \eqref{closeb}.
\end{proof}

\section{Non-simply normal numbers and examples}
\label{examples}

For $\beta\in(1,2]$ let
 $$\mathcal N_\beta:=\left\{x\in \Big(0,\frac{1}{\beta-1}\Big):x \textrm{ does not have a simply normal }\beta\textrm{-expansion}\right\}.$$ By Theorems \ref{JSS theorem} and \ref{Main theorem} it follows that $\mathcal N_\beta=\emptyset$ for any $\beta\in(1, \beta_T]$, and $\mathcal N_\beta\ne\emptyset$ for any $\beta\in(\beta_T, 2]$. Indeed, by   \cite[Lemma 2.3]{JSS} it follows that $\dim_H\mathcal N_\beta>0$ for any $\beta\in(\beta_T, 2]$. In \cite{Bak} the first author showed that $\dim_{H} \mathcal N_\beta\to 1$ as $\beta\to 2$. Furthermore, when $\beta=2$ it is a consequence of the well known work of Besicovich and Eggleston \cite{Bes,Egg}, and Borel \cite{Bor}, that $ \mathcal N_2$ is a Lebesgue null set of full Hausdorff dimension. 
In the following theorem we show that  the set $\mathcal N_\beta$  is  indeed   a Lebesgue null set for all $\beta\in(1,2)$.
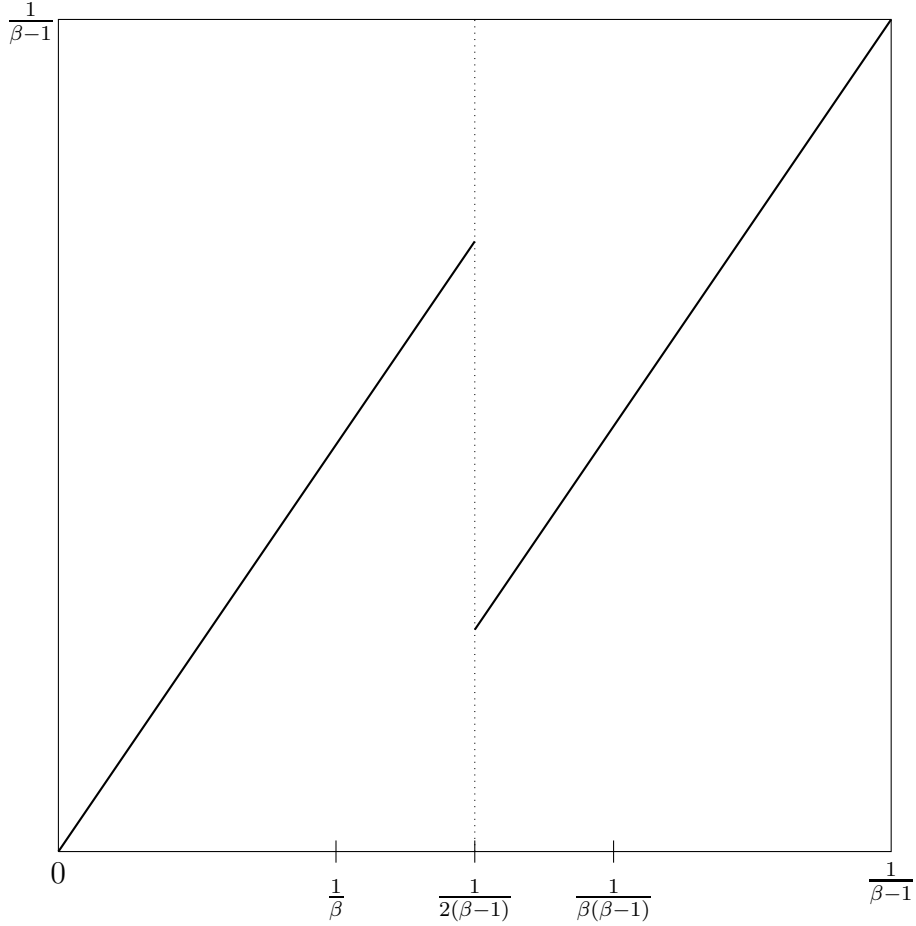
\begin{figure}[h!]
\centering
\begin{tikzpicture}[x=2.1,y=2.1]
\path[draw](0,10) -- (0,160) -- (150,160) -- (150,10) -- (0,10);
\path[draw][thick](0,10) -- (75,120);
\path[draw][thick](75,50) -- (150,160);
\path[draw](50,8) -- (50,12);
\path[draw](100,8) -- (100,12);
\path[draw](75,8) -- (75,12);
\path[draw][dotted](75,10)--(75,160);
\draw (0,10) node[below] {$0$};
\draw (50,7) node[below] {$\frac{1}{\beta}$};
\draw (100,7) node[below] {$\frac{1}{\beta(\beta-1)}$};
\draw (75,7) node[below] {$\frac{1}{2(\beta-1)}$};
\draw (150,10) node[below] {$\frac{1}{\beta-1}$};
\draw (-5,165) node[below] {$\frac{1}{\beta-1}$};
\end{tikzpicture}
\caption{The graph of $M_{\beta}$}
 \label{fig2}

\end{figure}

\begin{theorem}\label{th:measure}
Let $\beta\in(1,2)$. Then Lebesgue almost every $x\in I_{\beta}$ has a simply normal $\beta$-expansion.
\end{theorem}
\begin{proof}
Consider the following map $M_{\beta}:I_{\beta}\to I_{\beta}$:
\[M_{\beta}(x) = \left\{
  \begin{array}{lr}
    T_0(x)\, \textrm{ if }x\in[0,\frac{1}{2(\beta-1)})\\
    T_{1}(x)\, \textrm{ if }x\in[\frac{1}{2(\beta-1)},\frac{1}{\beta-1}].
  \end{array}
\right.
\]
We include a graph of the function $M_{\beta}$ in Figure \ref{fig2}. One can verify that the map $M_{\beta}$ eventually maps elements of $(0,\frac{1}{\beta-1})$ into the interval  $$A_{\beta}:=\Big[\frac{\beta}{2(\beta-1)}-1,\frac{\beta}{2(\beta-1)}\Big].$$ Moreover, once an element is mapped into $A_{\beta}$ it is never mapped out. The map $M_{\beta}$ is a piecewise linear expanding map, so we can employ the results of \cite{Kopf} and \cite{LiYo} to assert that there exists a unique $M_{\beta}$-invariant probability measure which is ergodic and absolutely continuous with respect to the Lebesgue measure. We call this measure $\mu$. We remark that as long as $x$ is never mapped onto the discontinuity point of $M_\beta$ then the following equality holds for all $n\in\mathbb{N}$:
\begin{equation}
\label{reflect}
M_{\beta}^n(x)=\frac{1}{\beta-1}-M_{\beta}^n\Big(\frac{1}{\beta-1}-x\Big).
\end{equation}In \cite{Kopf} the author gives an explicit formula for the density of $\mu.$ We do not state this formula here but merely remark that it is strictly positive on $A_{\beta}.$ This observation implies that there exists $x^*\in A_{\beta}$ such that its orbit under $M_{\beta}$ equidistributes in $A_{\beta}$ with respect to $\mu$, and the orbit of $\frac{1}{\beta-1}-x^*$ also equidistributes in $A_{\beta}$ with respect to $\mu$. Without loss of generality we may also assume that $x^*$ is not a preimage of the discontinuity point of $M_{\beta}$. Therefore, by the Birkhoff ergodic theorem and \eqref{reflect} we have
\begin{align*}
\mu\Big(\left[0,\frac{1}{2(\beta-1)}\right]\Big)&=\lim_{n\to\infty}\frac{1}{n}\sum_{k=0}^{n-1}\chi_{[0,\frac{1}{2(\beta-1)}]}M_{\beta}^k(x^*)\\
&=\lim_{n\to\infty}\frac{1}{n}\sum_{k=0}^{n-1}\chi_{[\frac{1}{2(\beta-1)},\frac{1}{\beta-1}]}M_{\beta}^k\left(\frac{1}{\beta-1}-x^*\right)\\
&=\mu\Big(\left[\frac{1}{2(\beta-1)},\frac{1}{\beta-1}\right]\Big).
\end{align*}

It follows therefore that $\mu([0,\frac{1}{2(\beta-1)}])=\mu([\frac{1}{2(\beta-1)},\frac{1}{\beta-1}])=1/2.$ Recall that we perform the map $T_0$ whenever an image point is in the interval $[0,\frac{1}{2(\beta-1)}),$ and we perform the map $T_1$ whenever our point is within the interval $[\frac{1}{2(\beta-1)},\frac{1}{\beta-1}].$ Consequently, by Lemma \ref{Bijection lemma} and the Birkhoff ergodic theorem, $\mu$ almost every $x$ has a simply normal $\beta$-expansion. Since $\mu$ has strictly positive density on $A_{\beta},$ it follows that Lebesgue almost every $x\in A_{\beta}$ has a simply normal $\beta$-expansion. Extending this statement to Lebesgue almost every $x\in I_{\beta}$ follows by considering preimages.
\end{proof}

Until  now the only elements we know in $\mathcal N_\beta$ are numbers with a unique $\beta$-expansion.  In the following we construct examples which show that there also exist  $\beta\in(\beta_T, 2]$ and $x\in(0, \frac{1}{\beta-1}),$ such that $x$ has precisely $k$ different $\beta$-expansions, and none of them are simply normal, where $k=2,3,\ldots,\aleph_0$ or $2^{\aleph_0}$. The following example was motivated by Erd\H{o}s and Jo\'{o} \cite{EJ}.
\begin{example}
Let $\beta\approx 1.92756$ be a multinacci number which is the root of $\beta^4-\beta^3-\beta^2-\beta-1=0$. Then $\al(\beta)=(1110)^\f$.  We claim that  for any $k\ge 1$
\[
x_k:=\pi_\beta(01^{4k-1}(011)^\f)
\]
has precisely $k$ different $\beta$-expansions. We will prove this by induction on $k$.

When $k=1$ we have $x_1=\pi_\beta(01^3(011)^\f)$. Then
\[
\overline{\al(\beta)}=(0001)\prec\si^n(01^3(011)^\f)\prec (1110)^\f=\al(\beta)
\]
for all $n\ge 0$.  By Lemma \ref{lem:univoque} it follows that  $x_1\in U_\beta$. Now suppose $x_k$ has precisely $k$ different $\beta$-expansions. We consider $x_{k+1}$. Since $\pi_\beta(10^\f)=\pi_\beta(01^40^\f)$, we have the word substitution $10^4\sim 01^4$. So,
\begin{align*}
x_{k+1}=\pi_\beta(01^{4k+3}(011)^\f)=\pi_\beta(10^41^{4k-1}(011)^\f)=\frac{1}{\beta}+\frac{x_k}{\beta^4}.
\end{align*}
By the inductive hypothesis it follows that $x_{k+1}$ has at least $k+1$ different $\beta$-expansions: one is $01^{4k+3}(011)^\f$ and the others begin with $10^3$. Furthermore, one can verify that $x_{k+1}$ has precisely $k+1$ different $\beta$-expansions by verifying that  $T_0(x_{k+1})\in U_\beta$ and $(T_0^i\circ T_1)(x_{k+1})\notin[\frac{1}{\beta}, \frac{1}{\beta(\beta-1)}]$ for all $i\in\set{0,1,2,3}$.

Therefore, $x_k$ has precisely $k$-different $\beta$-expansions, all of which  end with $(011)^\f$. Therefore, all $\beta$-expansions of $x_k$  are not simply normal. Letting $k\ra\f$ we conclude that $x_\f=\pi_\beta(01^\f)$ has a countable infinity of $\beta$-expansions, all of which end with $1^\f$, i.e., all $\beta$-expansions of $x_\f$ are not simply normal.
\end{example}

Now we construct an example of an $x$ which has  a continuum of $\beta$-expansions, none of which  are simply normal.
\begin{figure}[h!]
\begin{center}
  \centering
  \includegraphics[width=10cm]{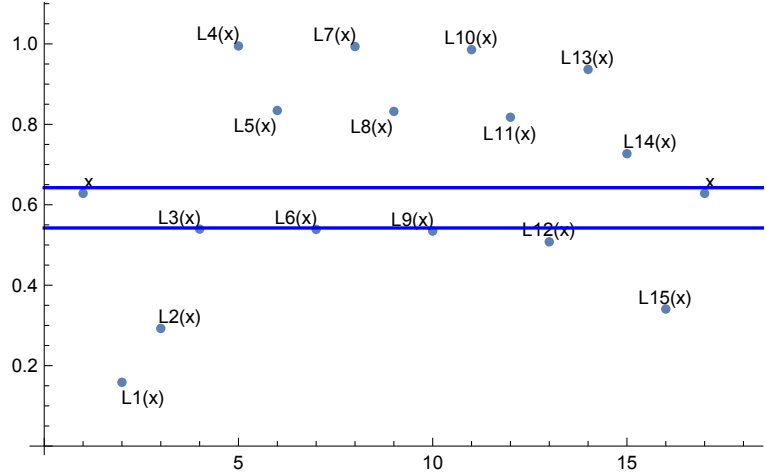}
\caption{The graph for the orbits $\set{L_i(x)}_{i=1}^{15}$. The region between the two horizontal lines is the switch region $[\frac{1}{\beta}, \frac{1}{\beta(\beta-1)}]$. }
\label{Fig4}
\end{center}
\end{figure}

\begin{figure}[h!]
\begin{center}
  \centering
  \includegraphics[width=10cm]{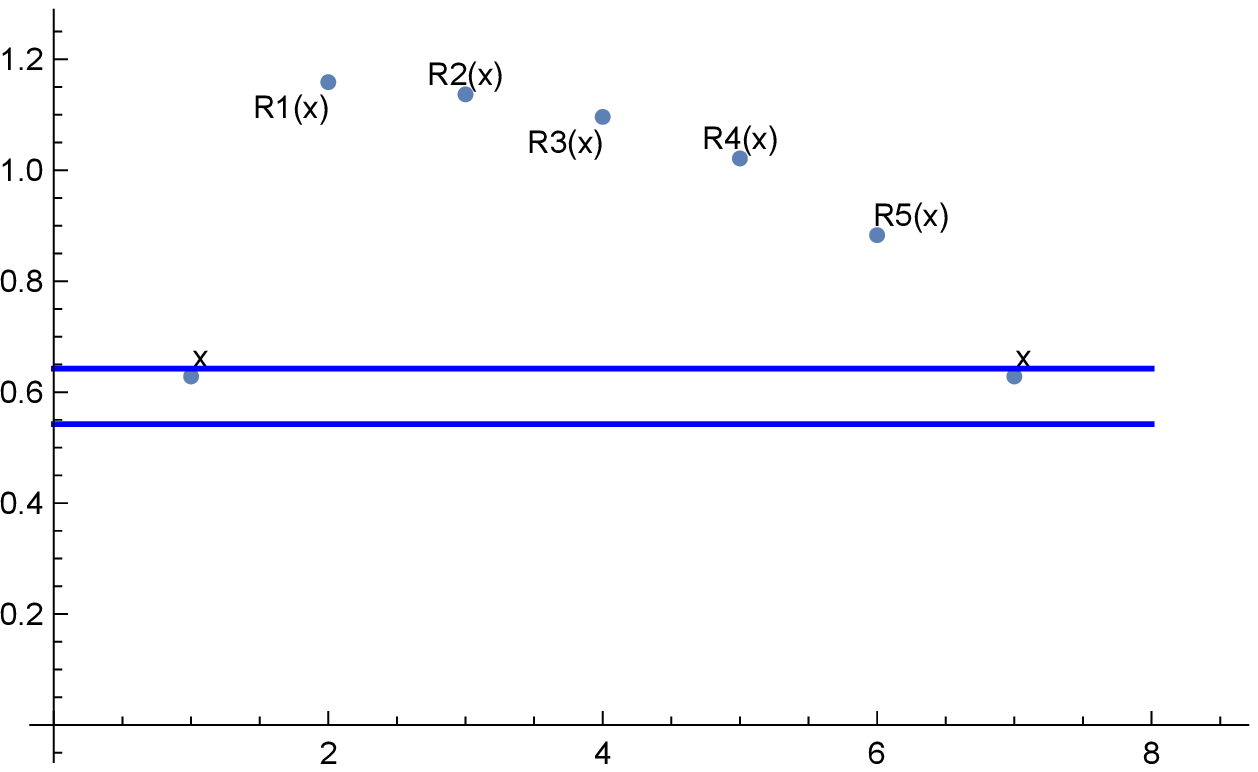}
\caption{The graph for the orbits $\set{R_i(x)}_{i=1}^5$. The region between the two horizontal lines is the switch region $[\frac{1}{\beta}, \frac{1}{\beta(\beta-1)}]$. }
\label{Fig5}
\end{center}
\end{figure}

\begin{example}
Let $\beta\approx 1.84408$ be the unique root in $(1,2]$ of
\[
\pi_\beta((10^3(110)^4)^\f)=\pi_\beta((01^5)^\f).
\]
By observing the substitution $10^3(110)^4\sim 01^5$ it follows that $x=\pi_\beta((01^5)^\f)\approx 0.628296$ has a continuum of $\beta$-expansions. We claim that all   $\beta$-expansions of $x$ are of the form
\begin{equation}\label{e41}
\mathbf x_1\mathbf x_2\cdots,
\end{equation}
where the words $\mathbf x_i=10^3(110)^4$ or $\mathbf x_i=01^5$ for all $i\ge 1$. Write $c_1\ldots c_{16}= 10^3(110)^4$ and $d_1\ldots d_6=01^5$. To prove this claim it suffices to show that the   orbits
\[
\set{L_i(x)=T_{c_1\ldots c_i}(x): i\in\set{1,\ldots, 15}}\quad\textrm{and}\quad \set{R_j(x)=T_{d_1\ldots d_j}(x):  j\in\set{1,\ldots, 5}}
\]
do not fall into the switch region $[\frac{1}{\beta}, \frac{1}{\beta(\beta-1)}]\approx [0.542276, 0.642445]$. This can be verified by some numerical calculation as described  in Figure \ref{Fig4} for the orbits $\set{L_i(x)}_{i=1}^{15}$ and in Figure \ref{Fig5} for the orbits $\set{R_i(x)}_{i=1}^5$.

Hence, all   $\beta$-expansions of $x$ are of the form in (\ref{e41}), and none of them are simply normal.
\end{example}

At the end of this section we pose some questions related to the set $\mathcal N_\beta$.
 In terms of Theorem \ref{th:measure}  it is natural to ask about the Hausdorff dimension of the set $\mathcal N_{\beta}$ for $\beta\in(\beta_T, 2)$.
\begin{itemize}
\item[{\bf Q1.}] For each $\beta\in(\beta_T, 2)$ can we calculate the Hausdorff dimension of $\mathcal N_\beta$?

\item [{\bf Q2.}] Is it true that $\dim_H\mathcal N_\beta<1$ for any $\beta<2$? This question was first raised in \cite{Bak}.

\item [{\bf Q3.}] Is the function $\beta\mapsto \dim_H\mathcal N_\beta$ continuous?
\end{itemize}

In this paper we study numbers with a simply normal $\beta$-expansion where $\beta\in(1,2]$ and the digit set is $\set{0, 1}$. It would be interesting to extend the results obtained in this paper to a larger digit set. To be more precise, study numbers with a simply normal $\beta$-expansion where $\beta\in(1,m+1]$ and the digit set is $\set{0, 1,\ldots, m}$ for some $m\in\mathbb N$. Denote by $\mathcal N_\beta(m)$ the set of all $x\in(0, \frac{m}{\beta-1})$ which do not have a simply normal $\beta$-expansion. We ask the following.
\begin{itemize}
\item[{\bf Q4.}] Does there exist a critical value $\beta_c=\beta_c(m)$ such that $\mathcal N_\beta(m)=\emptyset$ for any $\beta\in(1, \beta_c)$ and $\mathcal N_\beta(m)\ne\emptyset$ for any $\beta\in(\beta_c, m+1]$? Furthermore, if such a $\beta_c$ exists what can one say about $\mathcal {N}_{\beta_c}(m)$?

\item [{\bf Q5.}] What can we say about the Hausdorff dimension of   $\mathcal N_\beta(m)$ as in   {\bf Q1}--{\bf Q3}?
\end{itemize}

\section*{Acknowlegements}
The first author was supported by the EPSRC grant EP/M001903/1. The second author was supported by NSFC No.~11401516.

\end{document}